\documentclass[reqno]{amsart}

\usepackage{mathrsfs}
\usepackage{enumitem}
\usepackage{amssymb}
\usepackage{enumitem}
\usepackage{thm-restate}
\usepackage{tikz}
\usetikzlibrary{shapes.geometric}

\usepackage{nicefrac}
\usepackage[all]{xy}
\usepackage{todonotes}
\usepackage{color}
\usepackage{xcolor}

\usepackage[backref, hypertexnames=false]{hyperref}

\usepackage{comment}
\specialcomment{proofdetail}{\color{blue}}{\color{black}}

\excludecomment{proofdetail}

\theoremstyle{definition}
\newtheorem{thm}{Theorem}[section]
\newtheorem{cor}[thm]{Corollary}
\newtheorem{prop}[thm]{Proposition}

\newtheorem{defn}[thm]{Definition}
\newtheorem{ex}[thm]{Example}
\newtheorem{rmk}[thm]{Remark}

\newtheorem*{claim*}{Claim}

\numberwithin{equation}{section}

\newcommand\abcon{\ll}

\begin{document}

\title[Algorithmic randomness and the weak merging]
{Algorithmic randomness and the weak merging of computable probability measures}
\author{Simon M. Huttegger}

\address{Department of Logic and Philosophy of Science \\ 5100 Social Science Plaza \\ University of California, Irvine \\ Irvine, CA 92697-5100, U.S.A.}

\email{shuttegg@uci.edu}

\urladdr{http://faculty.sites.uci.edu/shuttegg/}

\author{Sean Walsh}

\address{Department of Philosophy \\ University of California, Los Angeles \\ 390 Portola Plaza, Dodd Hall 321 \\ Los Angeles, CA 90095-1451}

\email{walsh@ucla.edu}

\urladdr{http://philosophy.ucla.edu/person/sean-walsh/}

\author{Francesca Zaffora Blando}

\address{Department of Philosophy\\ Carnegie Mellon University \\Baker Hall 161\\
5000 Forbes Avenue\\ Pittsburgh, PA 15213}

\email{fzaffora@andrew.cmu.edu}

\urladdr{https://francescazafforablando.com/}



\date{\today}

\thanks{Many thanks to the reviewers and editors for helpful comments and feedback. Thanks also to audiences at the Berkeley Logic Colloquium; the Special Session on Computability at the AMS 2025 Fall Central Sectional Meeting; the Iowa Colloquium on Information, Complexity, and Logic; the University of Chicago Logic Seminar; the UCLA Logic Seminar; the UIC Logic Seminar; the Notre Dame Logic Seminar; and the Madison Logic Seminar.}

\begin{abstract}
We characterize Martin-L\"of randomness and Schnorr randomness in terms of the merging of opinions, along the lines of the Blackwell-Dubins Theorem \cite{Blackwell1962-ux}. After setting up a general framework for defining notions of merging randomness, we focus on finite horizon events: that is, on weak merging in the sense of Kalai-Lehrer \cite{Kalai1994-ta}. In contrast to Blackwell-Dubins and Kalai-Lehrer, we consider not only the total variational distance but also the Hellinger distance and the Kullback-Leibler divergence. Our main result is a characterization of Martin-L\"of randomness and Schnorr randomness in terms of weak merging and the summable Kullback-Leibler divergence. The main proof idea is that the Kullback-Leibler divergence between $\mu$ and $\nu$, at a given stage of the learning process, is exactly the incremental growth, at that stage, of the predictable process of the Doob decomposition of the $\nu$-submartingale $L(\sigma)=-\ln \frac{\mu(\sigma)}{\nu(\sigma)}$. These characterizations of algorithmic randomness notions in terms of the Kullback-Leibler divergence can be viewed as global analogues of Vovk's theorem \cite{Vovk1987} on what transpires locally with individual Martin-L\"of $\mu$- and $\nu$-random points and the Hellinger distance between $\mu,\nu$.
\end{abstract}

\maketitle

\setcounter{tocdepth}{4}
\setcounter{secnumdepth}{4}

\tableofcontents

\section{Introduction}\label{sec:intro}

Merging of opinions is a phenomenon that plays an important role in many fields, from the foundations of probability theory to Bayesian statistics, to economics and game theory, to machine learning. It subsumes a family of mathematical results that specify the conditions under which the predictions of different forecasters are guaranteed to become and stay close almost surely with increasing information. The main goal of this article is to study merging of opinions from a pointwise perspective using the tools of \emph{algorithmic randomness}\textemdash a branch of computability theory that offers a pointwise approach to computable measure theory and analysis \cite{HR2021}. More precisely, we propose a general framework for defining natural notions of \emph{merging randomness}: that is, notions of algorithmic randomness defined in terms of suitably effectivized notions of merging. Our main result, Theorem~\ref{thm:main:mlr:weak} below, provides a novel characterization of both Martin-L\"{o}f randomness and Schnorr randomness\textemdash two central algorithmic randomness notions\textemdash in terms of a weak form of merging. 

Arguably, the most well-known merging-of-opinions theorem in the literature is the Blackwell-Dubins Theorem \cite{Blackwell1962-ux}, which is concerned with predictions about infinite horizon events, and which takes the distance between successive probabilistic forecasts as given in terms of the \emph{total variational distance}. Kalai and Lehrer \cite{Kalai1994-ta} proposed  a notion of merging of opinions that only involves one-step-ahead predictions, which they called \emph{weak merging}. It is this weaker form of merging that features in our results. However, rather than merely focusing on the total variational distance like Kalai and Lehrer (and Blackwell and Dubins) did, here we also consider notions of merging randomness defined in terms of the \emph{Hellinger distance} and the \emph{Kullback-Leibler divergence}. In particular, our characterization of Martin-L\"{o}f randomness and Schnorr randomness via weak merging relies on the Kullback-Leibler divergence. Our results on strong merging in the sense of Blackwell and Dubins will appear in a companion article.

In this section we provide an overview of our results. The more elaborate proofs of the main results are given in later sections\textemdash in particular, the proofs of Theorem \ref{thm:main:mlr:weak}, Theorem \ref{thm:main:mlr:weakmiddle}, and Theorem \ref{thm:crzero}. A few short results and immediate consequences are given in this section as we proceed with the statements of the definitions and theorems.

The basic objects of study in this article are computable probability measures~$\nu$ on Cantor space $2^{\mathbb{N}}$ with full support. These are given by sequences of positive real numbers $\nu(\sigma)$, which are uniformly computable as $\sigma$ ranges over binary strings from $2^{<\mathbb{N}}$, and which satisfy the following defining properties of a probability measure:
\begin{equation*}
\nu(\emptyset)=1 \text{ and } \nu(\sigma 0)+\nu(\sigma 1)=\nu(\sigma),
\end{equation*}
where $\emptyset$ denotes the length-zero string and $\sigma$ again ranges over elements of $2^{<\mathbb{N}}$. Elements $\sigma$ of $2^{<\mathbb{N}}$ are functions $\sigma:\{0, \ldots, \ell-1\}\rightarrow \{0,1\}$, where $\ell = |\sigma|$ is the length of $\sigma$ (hence, the indexing starts at zero, and we write $\sigma= \sigma(0)\sigma(1)\cdots\sigma(\ell-1)$). By Carath\'eodory's Extension Theorem, all probability measures on Cantor space are induced by functions of this type by setting $\nu([\sigma])=\nu(\sigma)$, where $[\sigma]$ is the basic clopen $[\sigma]=\{\omega\in 2^{\mathbb{N}} : \forall \; i<\left|\sigma\right| \; \sigma(i)=\omega(i)\}$. We assume throughout that we are working with probability measures $\nu$ with \emph{full support}: that is, $\nu([\sigma])>0$ for all $\sigma$ in $2^{<\mathbb{N}}$. To streamline our notation, we often write $\nu(\sigma)$, rather than $\nu([\sigma])$.

For $n\geq 0$, we use $\mathscr{F}_n$ for the sub-$\sigma$-algebra of the Borel $\sigma$-algebra on Cantor space generated by the length-$n$ binary strings\textemdash that is, by the basic clopen events of the form $[\sigma]$, where $\sigma$ has length $n$. We use the following canonical versions of the conditional expectation and conditional probability:
\begin{equation}\label{eqn:preferredconditional}
\mathbb{E}_{\nu}[f\mid \mathscr{F}_n](\omega) = \frac{1}{\nu([\omega\upharpoonright n])} \int_{[\omega\upharpoonright n]} f \; d\nu, \hspace{10mm} \nu(A\mid \mathscr{F}_n)(\omega) = \nu(A\mid [\omega\upharpoonright n]),
\end{equation}
where $\omega\upharpoonright n$ denotes the initial segment of $\omega\in2^\mathbb{N}$ of length $n$.

For any probability measure $\nu$ on Cantor space and any sub-$\sigma$-algebra $\mathscr{G}$ of the Borel $\sigma$-algebra on Cantor space, we let $\nu\upharpoonright \mathscr{G}$ be the restriction of $\nu$ to events in $\mathscr{G}$. We combine this condition with the notation for the conditional probability by looking at the restriction $\nu(\cdot\mid \mathscr{F}_n)(\omega)\upharpoonright \mathscr{G}$ of the conditional probability $\nu(\cdot\mid \mathscr{F}_n)(\omega)$ to $\mathscr{G}$.

As standardly done, we use $\nu\abcon \mu$ to indicate that $\nu$ is \emph{absolutely continuous} with respect to~$\mu$. Recall that $\nu\abcon \mu$ if and only if, for all Borel events $A$, $\mu(A)=0$ implies that $\nu(A)=0$. Similarly, $\nu\upharpoonright \mathscr{G} \abcon \mu \upharpoonright \mathscr{G}$ holds if and only if, for all events $A$ in $\mathscr{G}$, one has that $\mu(A)=0$ implies that $\nu(A)=0$. When $\nu\abcon \mu$, the expression $\frac{d\nu}{d\mu}$ denotes the \emph{Radon-Nikodym derivative}; namely, the Borel measurable function such that $\nu(A)=\int_A \frac{d\nu}{d\mu} \; d\mu$ for all Borel events $A$, which is unique with this property up to $\mu$-a.s. equivalence. Similarly, when $\nu\upharpoonright \mathscr{G} \abcon \mu \upharpoonright \mathscr{G}$, the expression $\frac{d(\nu\upharpoonright \mathscr{G})}{d(\mu\upharpoonright \mathscr{G})}$ denotes the Radon-Nikodym derivative; namely, the $\mathscr{G}$-measurable function such that $\nu(A)=\int_A \frac{d\nu}{d\mu} \; d\mu$ for all events $A$ in $\mathscr{G}$, which is unique with this property up to $\mu\upharpoonright \mathscr{G}$-a.s. equivalence.

There are several important notions of information distance between probability measures that can be used to specify what it means for two probability measures to merge. As mentioned above, the three information distances we focus on in this article are the following:
\begin{defn}\label{defn:threedistance}
Suppose that $\mathscr{G}$ is a sub-$\sigma$-algebra of the Borel $\sigma$-algebra and that $\nu\upharpoonright \mathscr{G} \abcon \mu \upharpoonright \mathscr{G}$. Then we define:
\begin{enumerate}[leftmargin=*]
\item\label{defn:threedistance:1} \emph{Total variational distance}: $$T(\nu \upharpoonright \mathscr{G},\mu\upharpoonright \mathscr{G}):=\sup_{A\in \mathscr{G}} \left|\nu(A)-\mu(A)\right|$$
\item\label{defn:threedistance:2} \emph{Hellinger distance}: $$H(\nu\upharpoonright \mathscr{G}, \mu\upharpoonright \mathscr{G}) := \bigg(2\cdot \bigg(1-\mathbb{E}_{\mu\upharpoonright \mathscr{G}} \bigg[\frac{d(\nu\upharpoonright \mathscr{G})}{d(\mu\upharpoonright \mathscr{G})}^{\frac{1}{2}}\bigg] \bigg)\bigg)^{\frac{1}{2}}$$ which is often studied by recourse to the \emph{the Hellinger affinity}\footnote{Sometimes, the subscript ``$2$'' is put on the Hellinger affinity $\alpha$ (cf. \cite[p. 61]{Pollard2002-ih}). We avoid doing this here because we will later put other subscripts on the information distance symbols.} $$\alpha (\nu\upharpoonright \mathscr{G}, \mu\upharpoonright \mathscr{G}) := \mathbb{E}_{\mu\upharpoonright \mathscr{G}}   \bigg[\frac{d(\nu\upharpoonright \mathscr{G})}{d(\mu\upharpoonright \mathscr{G})}^{\frac{1}{2}}\bigg]$$ 
\item\label{defn:threedistance:3} \emph{Kullback-Leibler divergence}: $$D(\nu\upharpoonright \mathscr{G}\mid  \mu\upharpoonright \mathscr{G}) := \mathbb{E}_{\mu\upharpoonright \mathscr{G}} \bigg[\frac{d(\nu\upharpoonright \mathscr{G})}{d(\mu\upharpoonright \mathscr{G})} \ln \frac{d(\nu\upharpoonright \mathscr{G})}{d(\mu\upharpoonright \mathscr{G})}\bigg]=\mathbb{E}_{\nu\upharpoonright \mathscr{G}} \bigg[\ln \frac{d(\nu\upharpoonright \mathscr{G})}{d(\mu\upharpoonright \mathscr{G})}\bigg]$$
\end{enumerate}
\end{defn}
\noindent Of course, since $\mathbb{E}_{\nu\upharpoonright \mathscr{G}} [f]=\mathbb{E}_{\nu} [f]$ for all $\mathscr{G}$-measurable functions $f$, the restriction symbols can be dropped on expectations, which we will do in what follows. While it may not be obvious from the definition, one can show that the Kullback-Leibler divergence is non-negative.\footnote{This follows from Jensen's inequality applied to the convex function $x\ln x$ (see \cite[Theorem 2.3, p. 23]{Polyanskiy2025-xa}).}

While the total variational distance and the Hellinger distance are metrics, the Kullback-Leibler divergence is not even symmetric. Hence, with the Kullback-Leibler divergence, it is important to note the convention that the lower measure $\nu\upharpoonright \mathscr{G}$, in the ordering of absolute continuity, is listed first in the expression ${D(\nu\upharpoonright \mathscr{G}\mid  \mu\upharpoonright \mathscr{G})}$. 

These distance-like notions each have a rich history within information geometry, and they can be unified via Csisz\'ar's notion of $f$-divergences (see \cite{Csiszar1967-dk,Csiszar2004-tk} and \cite[Chapter 7]{Polyanskiy2025-xa}). However, we will not pursue this generalization here. More important for us is to note that the distances are related as follows, where the last is called Pinsker's inequality (cf. \cite{Gibbs2002-kr} and \cite[pp. 61-62]{Pollard2002-ih}):
\begin{align}
T\leq H, & \hspace{20mm} H^2\leq 2T, \notag \\
H^2\leq D, & \hspace{20mm} T^2 \leq \frac{1}{2} D \hspace{10mm} \label{eqn:theinequalities}
\end{align}
Further, if $d$ is any of these distances, then, for any $\sigma$-algebras $\mathscr{G}, \mathscr{H}$, one has:
\begin{equation}\label{eqn:donskerconsequence}
\mathscr{G}\subseteq \mathscr{H} \Longrightarrow d(\nu\upharpoonright \mathscr{G}, \mu\upharpoonright \mathscr{G})\leq d(\nu\upharpoonright \mathscr{H}, \mu\upharpoonright \mathscr{H})
\end{equation}
This is obvious in the case of the total variational distance, since we have defined it in terms of a supremum. For the Kullback-Leibler divergence, \eqref{eqn:donskerconsequence} comes from a supremum formulation known as the Donsker-Varadhan variational representation. For the Hellinger distance, it follows from a variational representation for $f$-divergences \cite[\S{7.13}]{Polyanskiy2025-xa}.

The notions of merging randomness we study in this article are relative to four parameters:
\begin{defn}
A \emph{merging quadruple} is a four-tuple~$(p, \preceq, \mathscr{G}_n,\rho)$ such that
\begin{enumerate}[leftmargin=*]
    \item $p$ is a \emph{merging exponent}: i.e., $p=0$ or $p\geq 1$;
    \item $\preceq$ is a \emph{merging relation}: i.e., $\preceq$ is a reflexive binary relation on a subset of the probability measures with full support;
    \item $\mathscr{G}_n$ is a \emph{merging horizon}: i.e., $\mathscr{G}_n$ is a sequence of sub-$\sigma$-algebras of the Borel $\sigma$-algebra such that $\mathscr{F}_{n+1}\subseteq \mathscr{G}_n$ for all $n\geq 0$;
    \item $\rho$ is a \emph{merging distance}: i.e., $\rho$ is one of the three information distances (total variational, Hellinger, or Kullback-Leibler) from Definition \ref{defn:threedistance}.  
\end{enumerate}
\end{defn}
A merging relation captures the sense in which two probability measures agree before any observations are made. The most prominent example is absolute continuity. Other merging relations are defined in Definition \ref{defn:twomergingpreorders}. Merging horizons, on the other hand, capture the types of events relative to which one can study the merging of probability measures. Concrete examples are given below in Definition \ref{defn:weakvsstrong}. Lastly, merging distances and merging exponents quantify distances between probability measures (see Definition \ref{defn:randomness:mr}).

\begin{defn}\label{defn:therv}
Given a merging horizon $\mathscr{G}_n$ and a merging distance $\rho$, we define the random variable:
\begin{equation}\label{eqn:thesequenced}
\rho_{\mathscr{G}_n}(\nu,\mu)(\omega)=\rho(\nu(\cdot \mid \mathscr{F}_n)(\omega) \upharpoonright \mathscr{G}_n,  \mu(\cdot \mid \mathscr{F}_n)(\omega)\upharpoonright \mathscr{G}_n)
\end{equation}
\end{defn}

In this, $\nu(\cdot \mid \mathscr{F}_n)(\omega) \upharpoonright \mathscr{G}_n$ again denotes $\nu(\cdot \mid \mathscr{F}_n)(\omega)$ restricted to $\mathscr{G}_n$ (and similarly for $\mu$). In the case where $\rho$ is the Kullback-Leibler divergence, we write $D_{\mathscr{G}_n}(\nu\mid \mu)(\omega)$ rather than $D_{\mathscr{G}_n}(\nu, \mu)(\omega)$. Even though it is not properly an information distance, we also use the notation from~(\ref{eqn:thesequenced}) for the Hellinger affinity and write $\alpha_{\mathscr{G}_n}(\nu,\mu)(\omega)$.

The following definition distinguishes between two types of merging horizons:

\begin{defn}\label{defn:weakvsstrong}
If $\mathscr{G}_n=\mathscr{F}_{n+1}$ for all $n\geq 0$, then the merging horizon is said to be \emph{weak}. If $\mathscr{G}_n$ is the Borel $\sigma$-algebra for all $n\geq 0$, then the merging horizon is said to be \emph{strong}. 
\end{defn}
\noindent As already explained above, in this article we mostly focus on weak merging and study strong merging in a companion article. Obviously, there are situations which fall in between, and towards the end of the article we discuss some of the middle ground (see \S\ref{sec:middleground}).

The primary merging relations we are interested in, in the context of weak merging, are defined as follows:

\begin{defn}\label{defn:twomergingpreorders}\;
\begin{enumerate}[leftmargin=*]
\item\label{defn:twomergingpreorders:0} $\nu\abcon_{kl} \mu$ if and only if  $\sup_n \mathbb{E}_{\nu} \ln \frac{\nu(\cdot \upharpoonright n)}{\mu(\cdot \upharpoonright n)}$ is finite;
\item\label{defn:twomergingpreorders:1} $\nu\abcon_{klc} \mu$ if and only if  $\sup_n \mathbb{E}_{\nu} \ln \frac{\nu(\cdot \upharpoonright n)}{\mu(\cdot \upharpoonright n)}$ is finite and computable;
\item $\nu\abcon_{bd}\mu$ if and only if $\mathbb{E}_{\nu}\sup_n \frac{\nu(\cdot \upharpoonright n)}{\mu(\cdot \upharpoonright n)}$ is finite;
\item $\nu\abcon_{bdc}\mu$ if and only if $\mathbb{E}_{\nu}\sup_n \frac{\nu(\cdot \upharpoonright n)}{\mu(\cdot \upharpoonright n)}$ is finite and computable;
\item\label{defn:twomergingpreorders:2} $\nu\abcon_{comp} \mu$ if and only if there is a computable function $m:\mathbb{Q}^{>0}\rightarrow \mathbb{Q}^{>0}$ such that, for all rationals $\epsilon>0$ and all Borel events $A$, if $\mu(A)<m(\epsilon)$, then $\nu(A)<\epsilon$.
\end{enumerate}
\end{defn}
\noindent While it is not obvious from the definitions, the expectations in (\ref{defn:twomergingpreorders:0})-(\ref{defn:twomergingpreorders:1}) are non-negative, since they are expectations of submartingales which are null at $n=0$ (cf. Example~\ref{ex:oursubmartingale}). One has the following implication diagram between the different merging relations discussed in this paper:
\footnotesize 
\begin{equation}\label{eqn:abcondiagram}
\xymatrix{\nu\abcon_{bd}\mu \ar@{->}^{Prop.~\ref{prop:bndimplieskl}}[r] & \nu\abcon_{kl}\mu \ar@{->}^{Prop.~\ref{prop:corofvovk}}[r] &  \nu\abcon_{\mathsf{MLR}}\mu    \ar@{->}^{Prop.~\ref{prop:shiryaev}}[r] & \nu\abcon\mu \\ 
 \nu\abcon_{bdc}\mu \ar@{->}[u]& \nu\abcon_{klc}\mu\ar@{->}[u] &  \nu\abcon_{comp}\mu  \ar@{->}[u]_{Lem.~\ref{lem:comptoMLR}} &
}
\end{equation}
\normalsize 
The notion $\nu\abcon_{comp} \mu$ is just the effectivization of the $\epsilon$-$\delta$ characterization of absolute continuity (see \cite[Definition 5]{hoyrup2011absolute} and \cite[Definition 2.14]{HR2021}). The notion $\nu\abcon_{\mathsf{MLR}}\mu$ is defined in Definition~\ref{defn:mlrcontain} below. One has that $\nu\abcon_{bdc}\mu$ and $\nu\abcon_{comp}\mu$ are entailed by  $\frac{d\nu}{d\mu}$ being a computable point of the computable Polish space $L_2(\mu)$ of square integrable functions (cf. Proposition~\ref{prop:iamgood}); but we do not know if a similar entailment holds for $\nu\abcon_{klc}\mu$. And we do not know whether, on the bottom row of Diagram~\ref{eqn:abcondiagram}, there are any implications from left to right (cf. Remark~\ref{rmk:below}). Since all the notions in Diagram~\ref{eqn:abcondiagram} imply $\nu\abcon \mu$, the $\nu$-expectations in Definition~\ref{defn:twomergingpreorders} can be turned into $\mu$-expectations by adding on the Radon-Nikodym derivative $\frac{d\nu}{d\mu}$. We prefer to express them as $\nu$-expectations since our randomness notation, to which we presently turn, is centered around $\nu$.

A merging quadruple induces a notion of merging randomness as follows:
\begin{defn}\label{defn:randomness:mr}
Suppose that $(p, \preceq, \mathscr{G}_n,\rho)$ is a merging quadruple. Suppose that $\nu$ is a computable probability measure with full support.
\begin{enumerate}[leftmargin=*]
\item If $p=0$, then $\omega$ is \emph{$\nu$-merging random with respect to the merging quadruple} ${(0,\preceq, \mathscr{G}_n,\rho)}$, abbreviated $\mathsf{MR}^{\nu}_0(\preceq, \mathscr{G}_n,\rho)$, if, for all computable probability measures $\mu$ with $\nu\preceq \mu$ and with full support, one has that $\lim_n \rho_{\mathscr{G}_n}(\nu,\mu)(\omega)=0$.
\item If $p\geq 1$, then $\omega$ is \emph{$\nu$-merging random with respect to the merging quadruple} ${(p,\preceq, \mathscr{G}_n,\rho)}$, abbreviated $\mathsf{MR}^{\nu}_p(\preceq, \mathscr{G}_n,\rho)$, if, for all computable probability measures $\mu$ with $\nu\preceq \mu$ and with full support, one has that $\sum_n (\rho_{\mathscr{G}_n}(\nu,\mu)(\omega))^p<\infty$.
\end{enumerate}
\end{defn}
\noindent As one can see from this definition, the purpose of the merging exponent is to track whether the sequence $\rho_{\mathscr{G}_n}(\nu,\mu)(\omega)$ of non-negative reals is in the Banach space $c_0$ of sequences $x_n$ satisfying $\lim_n x_n=0$, or whether the sequence $\rho_{\mathscr{G}_n}(\nu,\mu)(\omega)$ is in the Banach space $\ell^p$ of sequences $x_n$ satisfying $\sum_n \left|x_n\right|^p<\infty$. This is useful to keep track of, since convergence is sensitive to the value of $p\geq 1$. Sometimes, in what follows, to have fewer parentheses, we write $\rho_{\mathscr{G}_n}^p(\nu,\mu)(\omega)$ for $(\rho_{\mathscr{G}_n}(\nu,\mu)(\omega))^p$.

There are some obvious containment relations which are worth noting at the outset. First, since convergence of a series implies convergence of the associated limit to zero, one has:
\begin{equation}\label{eqn:sumconvergesimplieslimitgoestozero}
p\geq 1 \Longrightarrow \mathsf{MR}^{\nu}_p(\preceq, \mathscr{G}_n,\rho)\subseteq \mathsf{MR}^{\nu}_0(\preceq, \mathscr{G}_n,\rho)
\end{equation}
Further (\ref{eqn:donskerconsequence}) implies that these randomness notions are anti-monotonic in their parameters. In particular, for $p=0$, if $\preceq$ is a sub-relation of $\preceq^{\prime}$, and $\mathscr{G}_n$ is a subset of $\mathscr{G}_n^{\prime}$ for all $n\geq 0$, and $\rho\leq c\cdot \rho^{\prime}$ for some constant $c>0$, then
\begin{equation}
\mathsf{MR}^{\nu}_0(\preceq^{\prime}, \mathscr{G}_n^{\prime},\rho^{\prime}) \subseteq \mathsf{MR}^{\nu}_0(\preceq, \mathscr{G}_n,\rho)
\end{equation}
Likewise, for $p,p^{\prime}\geq 1$, if $\preceq$ is a sub-relation of $\preceq^{\prime}$, and $\mathscr{G}_n$ is a subset of $\mathscr{G}_n^{\prime}$  for all $n\geq 0$, and $\rho^p\leq c\cdot (\rho^{\prime})^{p^{\prime}}$ for some constant $c>0$, then
\begin{equation}
\mathsf{MR}^{\nu}_{p^{\prime}}(\preceq^{\prime}, \mathscr{G}_n^{\prime},\rho^{\prime} ) \subseteq \mathsf{MR}_p^{\nu}(\preceq, \mathscr{G}_n,\rho)
\end{equation}
To illustrate anti-monotonicity with respect to exponents $\geq 1$, the four inequalities in (\ref{eqn:theinequalities}) imply the following four containments:
\begin{align*}
\mathsf{MR}_1^{\nu}(\preceq, \mathscr{G}_n,H) & \subseteq  \mathsf{MR}_1^{\nu}(\preceq, \mathscr{G}_n,T) \hspace{5mm} & \mathsf{MR}_1^{\nu}(\preceq, \mathscr{G}_n,T)  \subseteq  \mathsf{MR}_2^{\nu}(\preceq, \mathscr{G}_n,H) \\
\mathsf{MR}_1^{\nu}(\preceq, \mathscr{G}_n,D) & \subseteq  \mathsf{MR}_2^{\nu}(\preceq, \mathscr{G}_n,H) \hspace{5mm} &\mathsf{MR}_1^{\nu}(\preceq, \mathscr{G}_n,D) \subseteq  \mathsf{MR}_2^{\nu}(\preceq, \mathscr{G}_n,T)
\end{align*}

As mentioned earlier, in this article we restrict attention to probability measures with full support (namely, those which satisfy  $\nu([\sigma])>0$ for all $\sigma$ in $2^{<\mathbb{N}}$). This is well-motivated by the interpretation of these probability measure as forecasting systems, and it dispenses us from needing to persistently qualify matters to rule out division by zero. It also implies $\nu\upharpoonright \mathscr{F}_n\abcon \mu\upharpoonright \mathscr{F}_n$ for all $n\geq 0$ for all those $\mu,\nu$ that we consider. That is to say, every pair of $\mu,\nu$ we consider are \emph{locally absolutely continuous}. Then, we can use the following special case of a classical theorem due to Kabanov, Lipcer and Shiryaev \cite{Kabanov1977-yr}, also contained in Shiryaev \cite[Theorem 4, pp. 169-171]{Shiryaev2019-rh}. Note that this is a classical result, and no assumptions are made about the computability of the measures.
\begin{restatable}{thm}{thmshiryaev}\label{thmshiryaev} Let $\nu$ and $\mu$ be two probability measures on Cantor space with full support. Then $\nu\abcon \mu$ if and only if, for $\nu$-a.s. many $\omega$, one has that $\sum_n H^2_{\mathscr{F}_{n+1}}(\mu,\nu)(\omega)<\infty$.
\end{restatable}
\noindent This is a special case of their general result, and it is only valid in Cantor space when $\nu,\mu$ have full support. We verify that this special case indeed follows from their more general result in \S\ref{sec:proof:hellinger}, when we first come to use it. Note that Theorem~\ref{thmshiryaev} has the following consequence:
\begin{cor}\label{corshiryaev}
Let $\preceq$ be a merging relation. If $\mathsf{MR}_2^{\nu}(\preceq, \mathscr{F}_{n+1}, H)$ has $\nu$-measure one and $\nu\preceq \mu$, then $\nu\abcon \mu$.
\end{cor}

Blackwell and Dubins  \cite{Blackwell1962-ux} showed the following fundamental result on strong merging:
\begin{thm} The set $\mathsf{MR}^{\nu}_0(\abcon, \mathsf{Borel},T)$ has $\nu$-measure one.
\end{thm}
\noindent From this and anti-monotonicity, one can deduce the following corollary about weak merging, which was studied in applied economic contexts by Kalai and Lehrer \cite{Kalai1994-ta}:
\begin{cor}
The set $\mathsf{MR}^{\nu}_0(\abcon, \mathscr{F}_{n+1},T)$ has $\nu$-measure one.
\end{cor}
\noindent Kalai and Lehrer \cite{KL1990, KL1993a, KL1993b, Kalai1994-ta} used these results to establish that, in an infinitely repeated game, rational players eventually approximate a Nash equilibrium, as long as this
equilibrium play is absolutely continuous with respect to the players' beliefs.\footnote{For a non-exhaustive list of articles spawned by Kalai and Lehrer's work on this topic, also see \cite{Nyarko1994}, \cite{Nachbar1997}, \cite{Nyarko1998}, \cite{Sandroni1998a}, \cite{Sandroni1998b}, \cite{Nachbar2005}, and \cite{Norman2022}.}

Our goal in this article is to study weak merging with respect to computable probability measures and the merging relations from Definition~\ref{defn:twomergingpreorders}. In our results, $\mathsf{MLR}^{\nu}$ is an abbreviation for the Martin-L\"of $\nu$-random points and $\mathsf{SR}^{\nu}$ is an abbreviation for the Schnorr $\nu$-random points. Recall that a \emph{c.e. open} $U$ is an event which can be written as $\bigcup_n [\sigma_n]$ for some comptuable sequence $\sigma_n$ of binary strings. Then $\omega$ is \emph{Martin-L\"of $\nu$-random} \cite{Martin-Lof1966aa}  (respectively, \emph{Schnorr $\nu$-random} \cite{Schn:1971, Schnorr1971aa}) if and only if, for every computable sequence $U_n$ of c.e. opens in Cantor space such that $\nu(U_n)\leq 2^{-n}$ for all $n\geq 0$ (respectively, such that $\nu(U_n)\leq 2^{-n}$ and $\nu(U_n)$ is computable, uniformly  in $n$, for all $n\geq 0$), one has that $\omega$ is not in $\bigcap_n U_n$. These sequences of c.e. opens $U_n$ are called \emph{sequential Martin-L\"of $\nu$-tests} (respectively, \emph{sequential Schnorr $\nu$-tests}).  

Our main theorem is the following. Its proof can be found in \S \ref{sec:proof:main}. 
\begin{restatable}{thm}{thmmainmlrweak}\label{thm:main:mlr:weak}
 Let $\nu$ be a computable probability measure on Cantor space with full support. Then $\mathsf{MLR}^{\nu}= \mathsf{MR}^{\nu}_1(\abcon_{kl}, \mathscr{F}_{n+1},D)$ and $\mathsf{SR}^{\nu}= \mathsf{MR}^{\nu}_1(\abcon_{klc}, \mathscr{F}_{n+1},D)$.
\end{restatable}

It is natural to ask what happens when $\mathscr{F}_{n+1}$ is replaced by $\mathscr{F}_{n+\ell}$ for $\ell>1$. We know the following, but we do not know whether the stated containment for Schnorr randomness can be strengthened to an identity. The proof can be found in~\S\ref{sec:middleground}. 
\begin{restatable}{thm}{thmmainmlrweakmiddle}\label{thm:main:mlr:weakmiddle}
Let $\nu$ be a computable probability measure on Cantor space with full support. Then, for $\ell>1$, one has the following:
$\mathsf{MLR}^{\nu}= \mathsf{MR}^{\nu}_1(\abcon_{kl}, \mathscr{F}_{n+\ell},D)$ and $\mathsf{SR}^{\nu}\supseteq \mathsf{MR}^{\nu}_1(\abcon_{klc}, \mathscr{F}_{n+\ell},D)$.
\end{restatable}

To round out the absolute continuity notions discussed in this article, we define:

\begin{defn}\label{defn:mlrcontain}
$\nu\abcon_{\mathsf{MLR}}\mu$ if $\mathsf{MLR}^{\nu}\subseteq \mathsf{MLR}^{\mu}$.
\end{defn}

The use of algorithmic randomness to study merging of opinions from a pointwise perspective has a notable predecessor in the work of Vovk \cite{Vovk1987,Vovk2009}, which was later extended by Fujiwara  \cite{Fujiwara2008}, who generalized Vovk's result (Theorem \ref{Vovk}) to the setting of all $\alpha$-divergences with $\alpha\in(-1, 1)$. Vovk's main theorem describes when a point is Martin-L\"{o}f random relative to two computable probability measures in terms of the convergence of the squared Hellinger distance:\footnote{When $\mu$ and $\nu$ are generalized Bernoulli measures, Vovk's result further yields an effective version of Kakutani's Theorem \cite{Kakutani1948-st}. We have included full support as a hypothesis on Vovk's Theorem to align it with our standing assumption. See \cite{Vovk1987} for a discussion of how to weaken this to the requirement that $\{\sigma: \nu(\sigma)=0\}$ be computable (and similarly for $\mu$). Likewise, \cite{Vovk1987} works in spaces $X^{\mathbb{N}}$ where $X\subseteq \mathbb{N}$ is computable, but to align the result with our discussion we restrict attention to Cantor space.}

\begin{thm}[\cite{Vovk1987}]\label{Vovk}
Let $\nu$ and $\mu$ be two computable probability measures on Cantor space with full support. Suppose that $\omega$ is in $\mathsf{MLR}^{\nu}$. Then $\omega$ is in $\mathsf{MLR}^{\mu}$ if and only if ${\sum_n H^2_{\mathscr{F}_{n+1}}(\nu, \mu)(\omega)}<\infty$.
\end{thm}

\noindent Crucially, as noted by Fujiwara, Theorem \ref{Vovk} does not hold when one replaces the squared Hellinger distance with the Kullback-Leibler divergence. But this is not in tension with our Theorem~\ref{thm:main:mlr:weak}, since Vovk's Theorem~\ref{Vovk} is a local result about specific pairs of sequences, whereas Theorem~\ref{thm:main:mlr:weak} is a global result characterizing the class of Martin-L\"{o}f $\nu$-random sequences in terms of merging.

However, we can make both directions of Vovk's Theorem~\ref{Vovk}  interact with global notions. We first note the following, which was observed by Hoyrup and Rojas \cite[\S{2}]{hoyrup2011absolute}:

\begin{restatable}{lem}{}\label{lem:comptoMLR}
    Let $\nu$ and $\mu$ be two computable probability measures on Cantor space with full support. Then $\nu \abcon_{comp}\mu$ implies $\nu\abcon_{\mathsf{MLR}}\mu$.
\end{restatable}

\begin{proof}
Let $m:\mathbb{Q}^{>0}\rightarrow \mathbb{Q}^{>0}$ be the computable function from the definition of $\nu\abcon_{comp}\mu$ and let $\ell:\mathbb{N}\rightarrow \mathbb{N}$ be a computable function such that $2^{-\ell(n)}<m(2^{-n})$ for all $n\geq 0$. Then any sequential Martin-L\"of $\mu$-test $U_n$ is such that $U_{\ell(n)}$ is a sequential Martin-L\"of $\nu$-test. 
\end{proof}

Using the forward direction of Vovk's Theorem~\ref{Vovk}, one can prove the following: 

\begin{restatable}{thm}{thmmlrtotalvar}\label{thm:main:totalvar}
 Let $\nu$ be a computable probability measure on Cantor space with full support. Then $\mathsf{MLR}^{\nu}\subseteq \mathsf{MR}^{\nu}_2(\abcon_{comp}, \mathscr{F}_{n+1},H)$.
\end{restatable}
\begin{proof}
Suppose that $\omega$ is in  $\mathsf{MLR}^{\nu}$ and that $\mu$ is a computable probability measure with full support such that $\nu\abcon_{comp}\mu$. By Lemma \ref{lem:comptoMLR}, we have that $\omega$ is in both $\mathsf{MLR}^{\nu}$ and $\mathsf{MLR}^{\mu}$. Hence, by the forward direction of Vovk's Theorem~\ref{Vovk}, one has that ${\sum_n H^2_{\mathscr{F}_{n+1}}(\nu, \mu)(\omega)}<\infty$.
\end{proof}
\noindent Further, the backward direction of Vovk's Theorem~\ref{Vovk} and our Theorem \ref{thm:main:mlr:weak} show the following about the notion $\abcon_{kl}$:
\begin{prop}\label{prop:corofvovk}
If $\nu\abcon_{kl}\mu$, then  $\nu \abcon_\mathsf{MLR} \mu$.
\end{prop}
\begin{proof}
Suppose that $\nu\abcon_{kl}\mu$ and $\omega$ is in $\mathsf{MLR}^{\nu}$. We show that $\omega$ in $\mathsf{MLR}^{\mu}$. By Theorem \ref{thm:main:mlr:weak}, we have that $\omega$ is in ${\mathsf{MR}^{\nu}_1(\abcon_{kl}, \mathscr{F}_{n+1},D)}$. Since $\nu\abcon_{kl}\mu$, we have ${\sum_n D_{\mathscr{F}_{n+1}}(\nu\mid\mu)(\omega)}<\infty$. Then, given that $H^2\leq D$, we have ${\sum_n H^2_{\mathscr{F}_{n+1}}(\nu, \mu)(\omega)}\leq {\sum_n D_{\mathscr{F}_{n+1}}(\nu\mid \mu)(\omega)}<\infty$.
Hence, by the backward direction of Vovk's Theorem \ref{Vovk}, we have that $\omega$ is in $\mathsf{MLR}^{\mu}$.
\end{proof}
\noindent Finally, the forward direction of Vovk's Theorem~\ref{Vovk} gives the forward containment in the following:
\begin{restatable}{cor}{corvovkish}\label{corvovkish}
$\mathsf{MLR}^{\nu}=\mathsf{MR}^{\nu}_2(\abcon_\mathsf{MLR}, \mathscr{F}_{n+1},H)$.  \end{restatable}
\noindent The backward containment is slightly more involved and is established at the close of \S\ref{sec:proof:hellinger}. Unlike Theorem \ref{thm:main:mlr:weak}, Martin-L\"of randomness plays into both sides of the equation in Corollary \ref{corvovkish}. The result can thus be thought of as establishing a kind of consistency between the set $\mathsf{MLR}^{\nu}$ and the relation $\abcon_\mathsf{MLR}$.

Finally, we turn to the case of having the exponent $p=0$, whose treatment requires an auxiliary notion. We say that $\omega$ is $\nu$-\emph{mild} if $\liminf_n \nu([\omega\upharpoonright (n+1)]\mid [\omega\upharpoonright n])>0$, and we abbreviate this as $\mathsf{Mild}^{\nu}$. Note that this corresponds to the notion of being conditionally bounded away from zero in \cite[Definition 5.2.3]{Li1997aa}. 

With this notion in hand, we have the following result:
\begin{restatable}{thm}{thmcrzero}\label{thm:crzero}
$\mathsf{Mild}^{\nu}\cap \mathsf{CR}^{\nu}\subseteq \mathsf{MR}_0^{\nu}(\abcon_{bdc}, \mathscr{F}_{n+1}, T)$.
\end{restatable}
\noindent In this, $\mathsf{CR}^{\nu}$ denotes the collection of computably $\nu$-random sequences, where \emph{computable randomness} is another central algorithmic randomness notion. It is defined in terms of the convergence of non-negative computable martingales, and we review this and other characterisations of $\mathsf{CR}^{\nu}$ in \S\ref{sec:cr}.

A natural follow-up question is whether Theorem~\ref{thm:crzero} can be reversed. Here we note the following:
\begin{rmk}\label{rmk:below}
Let $\nu$ be the uniform measure. If $\nu\abcon_{bdc}\mu$ implies $\nu\abcon_{klc}\mu$ for all $\mu$ (cf. Diagram~\ref{eqn:abcondiagram}, Proposition~\ref{prop:bndimplieskl}), then Theorem~\ref{thm:crzero} cannot be reversed. For, this hypothesis, anti-monotonicity, and Theorem~\ref{thm:main:mlr:weak} would imply that $$\mathsf{SR}^{\nu}=\mathsf{MR}_1(\abcon_{klc}, \mathscr{F}_{n+1}, D)\subseteq \mathsf{MR}_2(\abcon_{bdc}, \mathscr{F}_{n+1}, T)\subseteq \mathsf{MR}_0(\abcon_{bdc}, \mathscr{F}_{n+1}, T)\subseteq \mathsf{CR}^{\nu}$$ But this contradicts the fact that $\mathsf{CR}^{\nu}$ is a proper subset of $\mathsf{SR}^{\nu}$ for the uniform measure.  
\end{rmk}

The case of $p=0$ is closely related to one of the main concerns of Solomonoff's theory of induction, which was initiated in \cite{Solomonoff1964}. Solomonoff \cite{Solomonoff1978-ou} subsequently showed that if $\mu$ is a c.e. universal semi-measure\footnote{A \emph{c.e. semi-measure} $\mu$ is given by a sequence of uniformly left-c.e. reals $\mu(\sigma)>0$, as $\sigma$ ranges over $2^{<\mathbb{N}}$, such that $\mu(\emptyset)=1$ and $\mu(\sigma 0)+\mu(\sigma 1)\leq \mu(\sigma)$ for all $\sigma^{<\mathbb{N}}$. A c.e. semi-measure is \emph{universal} if, for every c.e. semi-measure (\emph{a fortiori}, for every computable probability measure) $\nu$, there is a constant $C>0$ such that $C\cdot \mu(\sigma)\geq \nu(\sigma)$ for all $\sigma$ in  $2^{<\mathbb{N}}$.} and $\nu$ is a computable probability measure, then for $\nu$-a.s. many $\omega$ in Cantor space, one has the following:
\begin{equation}\label{eqn:intro:solomonoffy}
\left|\mu(\omega\upharpoonright (n+1)\mid \omega\upharpoonright n)-\nu(\omega\upharpoonright (n+1)\mid \omega\upharpoonright n)\right|\rightarrow 0
\end{equation}
It is known that not every $\mathsf{MLR}^{\nu}$ is in this $\nu$-measure-one set, but no precise identification of this measure-one set in terms of algorithmic randomness has been discovered yet (see \cite[Theorem 6]{Hutter2004-iq} and \cite[375]{Li1997aa}). On the other hand, by Theorem~\ref{thm:crzero}, every element which is in both $\mathsf{MLR}^{\nu}$ and $\mathsf{Mild}^{\nu}$ satisfies the condition $\mathsf{MR}_0^{\nu}(\abcon_{bdc}, \mathscr{F}_{n+1}, T)$. And, as we note in \S\ref{sec:totalvariational}, for a sequence $\omega$ in $\mathsf{Mild}^{\nu}$, the condition of being in $\mathsf{MR}_0^{\nu}(\abcon_{bdc}, \mathscr{F}_{n+1}, T)$ is equivalent to (\ref{eqn:intro:solomonoffy}) happening for all computable probability measures $\mu$ with $\nu\abcon_{bdc} \mu$. In addition to being positioned differently in the hierarchy of randomness notions than Solmonoff's measure-one set, our merging notions are also more limited than Solomonoff's, and no claim is made to anything being a universal effective prior. Rather, what we get is information about merging with certain other computable probability measures reagrding which a given computable probability measure is absolutely continuous in the specified sense.

Solomonoff's theory of induction is an instance of the widely used Bayesian approach to scientific reasoning. The Bayesian approach requires agents with probabilistic beliefs to update their prior probability measures on observed data, but it leaves the choice of one's prior largely open. This raises the specter of arbitrariness: how can a scientific community ever achieve objective conclusions if inductive and statistical inference crucially depend on the particular priors chosen by the members of that community? Results on the merging of opinions provide an answer to these objections by showing that inter-subjective agreement (arguably, our best available proxy for objectivity) is guaranteed to occur within well-defined circumstances \cite{huttegger2015,Huttegger2017}. If the prior probabilities of the members of a community are \emph{initially sufficiently close}, then their posterior probabilities\textemdash that is, their prior probabilities conditioned on the data\textemdash will become closer and closer with increasing information. Thus, even if agents initially disagree in their probabilistic evaluations about an observational process, these disagreements are turned into a consensus as more data come in and the influence of their individual prior probabilities decreases. 

The classical sense in which probability measures are ``initially sufficiently close'' is absolute continuity. The notions of ``initially sufficiently close'' used in this paper, in Definition~\ref{defn:twomergingpreorders} and Figure~\ref{eqn:abcondiagram}, are sufficient conditions for absolute continuity expressed in terms of classical and effective features of the expectations of the ratios between the two probability measures. These conditions are easy to verify, since they just require checking that certain expressions are finite, or finite and computable. But they are also substantive, since they require different probability measures to share some basic modeling assumptions about the observational process at hand. In the light of these considerations, our main Theorem \ref{thm:main:mlr:weak} is especially significant. The random sequences of Cantor space with respect to some computable probability measure $\nu$, in the sense of $\mathsf{MLR}^{\nu}$ and $\mathsf{SR}^{\nu}$, are one way to make precise the informal idea of what the inductive assumptions of a Bayesian agent with prior $\nu$ are, since they are an expression of the effective statistical laws embedded in $\nu$ \cite{ZB2022}. From this perspective, Theorem \ref{thm:main:mlr:weak} says that, for a sequence of observational data, satisfying all of these effective statistical laws embedded into $\nu$ is equivalent to guaranteeing weak merging of opinions for all initially sufficiently close computable probability measures $\mu$. Not only does merging happen along all data streams which satisfy all effective statistical laws, but merging with sufficiently many other probability measures captures the same content as satisfying all of these effective statistical laws.

The results in this paper continue a tradition within the algorithmic randomness literature that develops effective versions of classical ``almost sure'' convergence results. Results on ``convergence to the truth'' are especially significant in the present context. These results establish when a sequence of random variables converges to a particular limit with probability one and, thus, provide another way of making precise the conditions under which a prior probability ceases its influence with increasing information. Many of the known effective versions of these classical theorems are anchored in Schnorr randomness \cite{Pathak2014aa,Rute2012aa,Huttegger2024-zb}, with Martin-L\"of randomness being conspicuously absent. By contrast, as shown in this article, weak merging has characterizations in terms of both Schnorr randomness and Martin-L\"of randomness. In that regard, the main result of this paper corresponds more closely to effective versions of Birkhoff's ergodic theorems, which also feature characertizations in terms of both Schnorr and Martin-L\"of randomness \cite{kucera1985measure,Bienvenu2012,Franklin2012,Franklin2014,Towsner2020}.

Finally, it is worth pointing out that the predictive setting of weak merging and merging for finite-horizon events is naturally aligned with several approaches in statistics, philosophy of science, and machine learning. In the philosophy of science, the paradigm of one-step-ahead prediction was at the center of Carnap's inductive logic program \cite{Carnap1950,Carnap1971,Carnap1980}. In statistics, predictive approaches have become increasingly influential \cite{Fong2023}. And much of machine learning consists in predicting finite batches of data based on a finite sequence of inputs \cite{Murphy2022-rr}.

This paper is organized as follows. In \S\ref{sec:dyadic} we briefly introduce some conventions on dyadic functions and recall the associated dyadic martingale notation. In \S\ref{sec:proof:main} we prove our main result Theorem~\ref{thm:main:mlr:weak} on the Kullback-Leibler divergence. An important part of the proof is Theorem~\ref{prop:klandtheincrements}, which shows that the Kullback-Leibler divergence $D_{\mathscr{F}_{n+1}}(\mu \mid\nu)(\omega)$ can be written as the increment between stage $n+1$ and $n$ of the predictable process from the Doob decomposition of the canonical $\nu$-submartingale $L_n(\omega)=-\ln\frac{\mu(\omega\upharpoonright n)}{\nu(\omega\upharpoonright n)}$ (cf. Example~\ref{ex:oursubmartingale}). In \S\ref{sec:proof:hellinger} we turn to the Hellinger distance and (i) show that Theorem~\ref{thmshiryaev} follows from the more general 
Kabanov-Lipcer-Shiryaev result, (ii) verify that
$\nu\abcon_{\mathsf{MLR}} \mu$ implies $\nu\abcon \mu$ in Proposition~\ref{prop:shiryaev}, and (iii) prove Corollary~\ref{corvovkish}. In \S\ref{sec:totalvariational} we turn to the total variational distance~$T$ and formulate some Cantor-space specific equivalent characterizations of the condition $\mathsf{MR}_0^{\nu}(\preceq, \mathscr{F}_{n+1}, T)$, which strongly resembles the limit (\ref{eqn:intro:solomonoffy}) from Solomonoff's theory. In \S\ref{sec:cr} we recall aspects of computable randomness $\mathsf{CR}^{\nu}$ and prove Theorem~\ref{thm:crzero}. Finally, in \S\ref{sec:middleground}, we prove what we know about slightly longer stretches of the horizon, namely Theorem~\ref{thm:main:mlr:weakmiddle}. We leave this material for last since changing the merging horizon introduces several distinctive complications.

\section{Dyadic functions and martingales}\label{sec:dyadic}

Since we work in Cantor space, it is often more convenient to work primitively with dyadic functions:
\begin{defn}[Dyadic functions]\label{defn:dyadic}
\text{}

A real-valued function is called \emph{dyadic} if it has domain $2^{<\mathbb{N}}$.  

A dyadic function $F:2^{<\mathbb{N}}\rightarrow \mathbb{R}$ is \emph{computable} if $F(\sigma)$ is a computable real number, uniformly in $\sigma$ from $2^{<\mathbb{N}}$.

If $F:2^{<\mathbb{N}}\rightarrow \mathbb{R}$ is a dyadic function, then $F_n:2^{\mathbb{N}}\rightarrow \mathbb{R}$ is given by $F_n(\omega)=F(\omega\upharpoonright n)$.
\end{defn}

Often, in what follows, we go back and forth between a dyadic function $F$ and the sequence of functions $F_n$ which it induces on Cantor space. For instance, the computability-theoretic properties are ascribed to $F$, while the measure-theoretic properties such as expectations are those of $F_n$. 

Here is another example. We say that a dyadic function $F:2^{<\mathbb{N}}\rightarrow \mathbb{R}$ is \emph{increasing} if $F(\sigma)\leq F(\tau)$ whenever $\sigma\preceq \tau$. Then, $F:2^{<\mathbb{N}}\rightarrow \mathbb{R}$ is increasing if and only if $F_n\leq F_{n+1}$ everywhere for all $n\geq 0$.

As another case in point, recall the traditional definition of a martingale from algorithmic randomness:
\begin{defn}[Dyadic martingales and related concepts]\label{defn:dyadicmartingale}
\text{}

A \emph{dyadic $\nu$-martingale} $M$ is a function $M:2^{<\mathbb{N}}\rightarrow \mathbb{R}$ satisfying the following, for all $\sigma$ in $2^{<\mathbb{N}}$,
\begin{equation}\label{eqn:martingalecondition}
M(\sigma) = M(\sigma 0)\nu(\sigma 0\mid \sigma)+M(\sigma 1)\nu(\sigma 1 \mid \sigma)
\end{equation}
\noindent Recall that $\nu$ is by assumption a full-support probability measure, so (\ref{eqn:martingalecondition}) is well-defined. {\it Submartingales} are defined the same, but with equality replaced by $\leq$; and {\it supermartingales} are defined just the same but with the equality replaced by $\geq$.
\end{defn}

Using the preferred version of the conditional expectation from (\ref{eqn:preferredconditional}), one can rephrase the martingale condition~(\ref{eqn:martingalecondition}) as $M_n = \mathbb{E}_{\nu}[M_{n+1}\mid \mathscr{F}_n]$, which matches the classical definition. Likewise, the submartingale condition can be rephrased as $M_n\leq \mathbb{E}_{\nu}[M_{n+1}\mid \mathscr{F}_n]$ and the supermartingale condition can be rephrased as $M_n\geq \mathbb{E}_{\nu}[M_{n+1}\mid \mathscr{F}_n]$.

Martingales are formalizations of betting strategies, but there are also natural examples of martingales throughout measure-theoretic mathematics:
\begin{ex}\label{ex:martingale:1}
 $M(\sigma)=\frac{\mu(\sigma)}{\nu(\sigma)}$ is a computable dyadic $\nu$-martingale. Part of the reason this martingale is important is that its ratios are likelihood ratios:
\begin{equation}\label{rmk:martingalebd:eqn:redeux}
 \frac{M_n(\omega)}{M_{n+1}(\omega)}= \frac{M(\omega\upharpoonright n)}{M(\omega\upharpoonright (n+1))} = \frac{\frac{\mu( \omega\upharpoonright n)}{\nu( \omega\upharpoonright n)}}{\frac{\mu( \omega\upharpoonright (n+1))}{\nu( \omega\upharpoonright (n+1))}} = \frac{ \nu( \omega \upharpoonright (n+1) \mid  \omega \upharpoonright n)}{ \mu( \omega \upharpoonright (n+1) \mid  \omega \upharpoonright n)}
\end{equation}

\end{ex}

Another natural class of effective martingales comes from considering the conditional expectations of lower-semicomputable (lsc) functions, where  function $f:2^{\mathbb{N}}\rightarrow [0,\infty]$ is \emph{lower semi-computable} if $f^{-1}(q,\infty]$ is c.e. open, uniformly in $q\in \mathbb{Q}^{\geq 0}$. 
\begin{ex}\label{ex:martingale:2}
 Suppose that $f:X\rightarrow [0,\infty]$ is lsc and has computable $\nu$-expectation. Then, $M(\sigma) = \mathbb{E}_{\nu}[f\mid \mathscr{F}_{\left|\sigma\right|}] (\sigma \overline{0})$ is a computable dyadic $\nu$-martingale. Here, $\sigma\overline{0}$ denotes the element of $2^{\mathbb{N}}$ that appends an infinite sequence of zeroes to $\sigma$.
\end{ex}

The main source of submartingales is Jensen's inequality. The second assertion in the following result notes an easy effectivization of this fact.

\begin{prop}\label{prop:jensensubmartingal}
Suppose that $M:2^{<\mathbb{N}}\rightarrow \mathbb{R}$ is a dyadic $\nu$-martingale which has range contained in a finite or infinite open interval $I$ of the reals. Suppose that $\varphi:I\rightarrow \mathbb{R}$ is convex. Then, $\varphi \circ M: 2^{<\mathbb{N}}\rightarrow \mathbb{R}$ is a dyadic $\nu$-submartingale. Moreover, if (i) $\varphi:I\rightarrow \mathbb{R}$ is computable continuous in addition to being convex, (ii) the endpoints of $I$ are in $\mathbb{Q}\cup \{-\infty,\infty\}$ and (iii) $M$ is a computable dyadic $\nu$-martingale, then $\varphi \circ M: 2^{<\mathbb{N}}\rightarrow \mathbb{R}$ is a computable dyadic $\nu$-submartingale.
\end{prop}

Condition (ii) is a simple sufficient condition for making sure that the interval $I$ is a computable Polish space. This fits with the notion of a computable continuous function. A function is said to be \emph{computable continuous} if inverse images of c.e. opens are uniformly c.e. open \cite{HR2021}.

\begin{proof}
Jensen's inequality shows that $\varphi(M)$ is a dyadic $\nu$-submartingale. Since $\varphi:I\rightarrow \mathbb{R}$ is computable continuous, it maps computable reals to computable reals. Hence, $\varphi(M(\sigma))$ is uniformly computable.
\end{proof}

Here is the main submartingale we use in this paper:
\begin{ex}\label{ex:oursubmartingale}
Let $M(\sigma)=\frac{\mu(\sigma)}{\nu(\sigma)}$ be a computable dyadic $\nu$-martingale. Then, $L(\sigma)=-\ln M(\sigma)$ is a computable dyadic $\nu$-submartingale such that $\mathbb{E}_{\nu} L_n\geq 0$ for all $n\geq 0.$

It is a computable dyadic $\nu$-submartingale by Proposition~\ref{prop:jensensubmartingal}, Example~\ref{ex:martingale:1}, and the fact that the negative logarithm is convex and computable continuous from $(0,\infty)$ to the reals. Further, submartingales $L$ satisfy $\mathbb{E}_{\nu} L_{n+1}\geq \mathbb{E}_{\nu} L_n$ for all $n\geq 0$. Since $L_0=0$, we have that $\mathbb{E}_{\nu} L_n\geq 0$ for all $n\geq 0$.

The increments of this submartingale are log-likelihood ratios:
\begin{equation}\label{rmk:martingalebd:eqn2}
L_{n+1}(\omega)-L_n(\omega)= \ln \frac{M(\omega\upharpoonright n)}{M(\omega\upharpoonright (n+1))} = \ln \frac{ \nu( \omega \upharpoonright (n+1) \mid  \omega \upharpoonright n)}{ \mu( \omega \upharpoonright (n+1) \mid  \omega \upharpoonright n)}
\end{equation}

\end{ex}

This submartingale is negative for many values, whereas in algorithmic randomness attention is often restricted to non-negative martingales and supermartingales. 

Finally, we note the following equivalent characterization of  $\nu\abcon_{kl} \mu$:
\begin{rmk}
Let $M(\sigma)=\frac{\mu(\sigma)}{\nu(\sigma)}$ be the dyadic $\nu$-martingale and let $L(\sigma)=-\ln M(\sigma)$ be the corresponding dyadic $\nu$-submartingale. Then, $\nu\abcon_{kl} \mu$ if and only if $\sup_n \mathbb{E}_{\nu} L_n<\infty$.
\end{rmk}
\noindent This is the formulation of $\nu\abcon_{kl} \mu$ which we use most often. Since we work with this so much, it is useful to take a moment to note the placement of $\nu(\sigma),\mu(\sigma)$ in the denominator and numerator of the expressions $M(\sigma)=\frac{\mu(\sigma)}{\nu(\sigma)}$ and $L(\sigma)=-\ln M(\sigma) = \ln \frac{\nu(\sigma)}{\mu(\sigma)}$.

\section{Kullback-Leibler Divergence and Proof of Theorem~\ref{thm:main:mlr:weak}}\label{sec:proof:main}

We begin by calculating some derivatives and then calculate $ D_{\mathscr{F}_{n+1}}(\nu \mid \mu)(\omega)$, where this is as in Definition~\ref{defn:therv}. Here and in what follows, we use $\iota$ and $\jmath$ as variables ranging over $\{0,1\}$. Further, if $\sigma$ is a binary string in $2^{<\mathbb{N}}$ of length $n$, then $\sigma\iota$ is the extension of $\sigma$ to a length $n+1$ binary string whose final entry is $\iota$. As is standard, if $\omega$ is an element of Cantor space $2^{\mathbb{N}}$ and $n\geq 0$, then $\omega\upharpoonright n$ is the element $\sigma$ of $2^{<\mathbb{N}}$ of length $n$ such that $\omega(i)=\sigma(i)$ for all $i<n$.
\begin{prop}\label{prop:derivativecalc}
One has:
\begin{equation}
\frac{d(\nu(\cdot \mid \mathscr{F}_n)(\omega) \upharpoonright \mathscr{F}_{n+1})}{d(\mu(\cdot \mid \mathscr{F}_n)(\omega) \upharpoonright \mathscr{F}_{n+1})}(\omega^{\prime}) = \sum_{\iota\in \{0,1\}} \frac{\nu(\omega\upharpoonright n\,\iota\mid \omega\upharpoonright n\,)}{\mu(\omega\upharpoonright n\,\iota\mid \omega\upharpoonright n\,)} \cdot X_n^{\iota}(\omega^{\prime}) \label{eqn:thedervweak}
\end{equation}
wherein $X_n^{1}(\omega)=\omega(n)$ and $X_n^{0}(\omega)=1-\omega(n)$. 
\end{prop}
Before we begin the proof, note that the random variables $X_n^{\iota}$ are defined so that $X_n^{\iota}(\omega)=1$ if and only if $\omega(n)=\iota$. Further, note that, to avoid an excess of parentheses, we write $\omega\upharpoonright n\,\iota$ for $(\omega\upharpoonright n)\,\iota$\textemdash that is, for the first $n$ bits of $\omega$ followed by $\iota$.
\begin{proof}
To show~(\ref{eqn:thedervweak}), we show that the expression on the right-hand side of it satisfies the defining property of the Radon-Nikodym derivative on the left-hand side of it. As an initial step, consider an arbitrary Borel event $A$, and consider the following, wherein $\frac{d(\nu(\cdot \mid \mathscr{F}_n)(\omega) \upharpoonright \mathscr{F}_{n+1})}{d(\mu(\cdot \mid \mathscr{F}_n)(\omega) \upharpoonright \mathscr{F}_{n+1})}(\omega^{\prime})$ means the formula on the right-hand side of (\ref{eqn:thedervweak}):
\begin{align}\label{eqn:derivative}
\int_A &\frac{d(\nu(\cdot \mid \mathscr{F}_n)(\omega) \upharpoonright \mathscr{F}_{n+1})}{d(\mu(\cdot \mid \mathscr{F}_n)(\omega) \upharpoonright \mathscr{F}_{n+1})}(\omega^{\prime}) \;d\mu(\cdot \mid \omega\upharpoonright n\,)(\omega^{\prime}) \notag \\ &= \sum_{\iota\in \{0,1\}} \nu(\omega\upharpoonright n\,\iota\mid \omega\upharpoonright n\,) \mu(A\mid [\omega\upharpoonright n\,\iota]) 
\end{align}
In the case where $A$ is in $\mathscr{F}_{n+1}$, then we have four options for $A\cap [\omega\upharpoonright n\,]$: it is empty or it is $[\omega\upharpoonright n\, 0]$ or it is $[\omega\upharpoonright n\, 1]$ or it is $[\omega\upharpoonright n\,]$. In all four cases, one sees that (\ref{eqn:derivative}) is equal to $\nu(A\mid \mathscr{F}_n)(\omega)=\nu(A\mid \omega\upharpoonright n)$.
\end{proof}

\begin{prop}\label{prop:calculatekullback}
One has:
\begin{equation}
 D_{\mathscr{F}_{n+1}}(\nu \mid \mu)(\omega)= \sum_{\iota\in \{0,1\}} \bigg(\ln \frac{\nu(\omega\upharpoonright n\, \iota \mid \omega\upharpoonright n\,)}{\mu(\omega\upharpoonright n\, \iota \mid \omega\upharpoonright n\,)} \bigg) \cdot \nu(\omega\upharpoonright n\, \iota \mid \omega\upharpoonright n\,)\label{eqn:dn1}
\end{equation}
\end{prop}
\begin{proof}
The equality follows immediately from (\ref{eqn:thedervweak}).
\end{proof}

We now note an effectivization of Doob's Decomposition Theorem \cite[120-121]{Williams1991aa} for dyadic submartingales. For a more general effectivization of the result, see \cite[Proposition 8.4]{Rute2012aa}. For our purposes in this article, we need to work with the explicit dyadic versions. 
\begin{thm}[Effective Dyadic Doob Decomposition]\label{thm:doobdecomp}
\text{}

Suppose that $L$ is a computable dyadic $\nu$-submartingale. Then there is a unique pair of computable dyadic $\nu$-martingale $N$ and increasing non-negative computable dyadic function $A$ with the following properties, where $\tau$ ranges over $2^{<\mathbb{N}}$ and $\iota$ ranges over $\{0,1\}$:
\begin{enumerate}[leftmargin=*]
    \item \label{eqn:apartofdoob:0} $L(\tau)=N(\tau)+A(\tau)$
    \item \label{eqn:apartofdoob:00} $N(\emptyset)=L(\emptyset)$
    \item \label{eqn:apartofdoob} $A(\tau\iota)-A(\tau) = \big(L(\tau 0) \nu(\tau 0 \mid \tau)+L(\tau 1)\nu(\tau 1\mid \tau) \big) - L(\tau)$
\end{enumerate}
\end{thm}
In the context of Doob's decomposition theorem, $A$ is known as the \emph{predictable process}, since (\ref{eqn:apartofdoob}) shows that $A(\tau\iota)$ depends on $\tau$ and not~$\iota$. Note that the identity in (\ref{eqn:apartofdoob}) can also be expressed in classical terms as follows, where $\omega$ ranges over $2^{\mathbb{N}}$ and $n\geq 0$:
\begin{equation}\label{eqn:classicalapartofdoob}
A_{n+1}(\omega)-A_n(\omega) = \mathbb{E}_{\nu}[L_{n+1}\mid \mathscr{F}_n](\omega) - L_n(\omega)
\end{equation}

\begin{proof}
Uniqueness of $A$ follows by an induction on the length of $\tau$, with the base case being taken care of by (\ref{eqn:apartofdoob:0})-(\ref{eqn:apartofdoob:00}) and the induction step being taken of by  (\ref{eqn:apartofdoob}); and then uniqueness of $N$ follows from (\ref{eqn:apartofdoob:0}). For existence, we first define an auxiliary computable dyadic function $H$:
\begin{equation*}
H(\emptyset) = L(\emptyset), \hspace{5mm} H(\tau\iota) = L(\tau\iota)-\big(L(\tau 0) \nu(\tau 0 \mid \tau)+L(\tau 1)\nu(\tau 1\mid \tau) \big)
\end{equation*}
Since $\iota$ does not appear after the minus sign in $H(\tau \iota)$, we have the following:
\begin{equation}\label{eqn:hproperty}
H(\tau 0)\nu(\tau 0 \mid \tau) +H(\tau 1)\nu(\tau 1 \mid \tau)=0
\end{equation}
Then we define computable dyadic functions $N,A$ by:
\begin{equation*}
N(\tau)=\sum_{i\leq \left|\tau\right|} H(\tau\upharpoonright i), \hspace{10mm} A(\tau)=L(\tau)-N(\tau)
\end{equation*}
By (\ref{eqn:hproperty}), $N$ is a dyadic $\nu$-martingale. Using this and the fact that $L$ is a dyadic $\nu$-submartingale, one obtains $A(\tau\iota)-A(\tau) \geq 0$.
\end{proof}

The following theorem shows how to express the Kullback-Leibler divergence in terms of the \emph{increment} of the predictable process from the Doob decomposition of the submartingale from Example~\ref{ex:oursubmartingale}.
\begin{thm}\label{prop:klandtheincrements}
Let $L(\tau)=-\ln \frac{\mu(\tau)}{\nu(\tau)}$, and let $A$ be the predictable process from the Doob decomposition of $L$. Then, one has:
\begin{equation}\label{eqn:klsub}
 D_{\mathscr{F}_{n+1}}(\nu \mid \mu)(\omega)   = A(\omega\upharpoonright (n+1)\,)-A(\omega\upharpoonright n\,)
\end{equation}
\end{thm}
\begin{proof}
To see this, first note for $M(\tau) = \frac{\mu(\tau)}{\nu(\tau)}$:
\begin{equation*}
    \ln \frac{\nu(\omega\upharpoonright n\, \iota \mid \omega\upharpoonright n\,)}{\mu(\omega\upharpoonright n\, \iota \mid \omega\upharpoonright n\,)} = - \ln M(\omega \upharpoonright n \iota) + \ln M(\omega \upharpoonright n) = L(\omega\upharpoonright n\,\iota)-L(\omega\upharpoonright n\,).
\end{equation*}
Then, by (\ref{eqn:dn1}), we have:
\begin{align*}
D_{\mathscr{F}_{n+1}}&(\nu \mid \mu)(\omega)  \\
&=  L(\omega\upharpoonright n\, 0)\nu(\omega\upharpoonright n\,0 \mid \omega\upharpoonright n\,)+ L(\omega\upharpoonright n\, 1)\nu(\omega\upharpoonright n\, 1 \mid \omega\upharpoonright n\,)-L(\omega\upharpoonright n\,).   
\end{align*}
From this and Theorem~\ref{thm:doobdecomp}(\ref{eqn:apartofdoob}) one gets the desired conclusion (\ref{eqn:klsub}).
\end{proof}

In this and the next proposition, we focus on turning lower semi-computable functions (abbreviated as \emph{lsc} functions) into predictable processes. 

Our first proposition is a simple ``predictive process'' variant of the usual ``from below'' representations of lsc functions (see \cite[315]{Li1997aa} or \cite[Lemma 5.6]{Miyabe2013-fd}).
\begin{prop}\label{prop:frommlrtest2athing}
Suppose that $f:2^{\mathbb{N}}\rightarrow [0,\infty]$ is lsc. 

Then, there is an increasing computable dyadic function $A:2^{<\mathbb{N}}\rightarrow \mathbb{Q}^{\geq 0}$ such that $A(\emptyset)=0$, $A(\sigma 0)=A(\sigma 1)$ for all $\sigma$ in $2^{<\mathbb{N}}$, and $f(\omega)=\sup_n A(\omega\upharpoonright n)$. 

Further, for any rational $\epsilon>0$, one can construct such an $A$ with the further property that $A(\sigma \iota)-A(\sigma)<\epsilon$ for each $\sigma$ in $2^{<\mathbb{N}}$ and each $\iota$ in $\{0,1\}$.

Finally, for any probability measure $\nu$, one has that $\sup_n \mathbb{E}_{\nu} A_n =\mathbb{E}_{\nu} f$.
\end{prop}
\begin{proof}
Let $q_m$ be a computable sequence which enumerates $\mathbb{Q}^{\geq 0}$ with $q_0=0$. Let $\tau_{m,i}$ be a computable sequence in $2^{<\mathbb{N}}$ such that $f^{-1}(q_m,\infty]=\bigsqcup_i [\tau_{m,i}]$ for all $m\geq 0$. Define $A(\emptyset)=0$ and 
\begin{equation*}
A(\sigma \iota)=\max(A(\sigma),\max\{q_m: q_m <A(\sigma)+\epsilon \; \& \; m\leq \left|\sigma\right| \; \& \;  \exists \; i\leq \left|\sigma\right| \; \sigma \succeq \tau_{m,i}\})
\end{equation*}
Then by definition $A(\sigma)\leq A(\sigma\iota)$ and so $A$ is increasing. To see that $f(\omega)=\sup_n A(\omega\upharpoonright n)$:
\begin{enumerate}[leftmargin=*]
\item Suppose $f(\omega)<\sup_n A(\omega\upharpoonright n)$. Since $A(\emptyset)=0$, there is least $n\geq 0$ such that $f(\omega)<A(\omega\upharpoonright (n+1))$. Let $A(\omega\upharpoonright (n+1))=q_m$ and let $\sigma=\omega\upharpoonright n$, so that $q_m<A(\sigma)+\epsilon$ and $m\leq \left|\sigma\right|$ and for some $i\leq \left|\sigma\right|$ we have $\sigma \succeq \tau_{m,i}$. Then $\omega$ is in $[\tau_{m,i}]\subseteq f^{-1}(q_m,\infty]$, and so $f(\omega)>q_m$, a contradiction.
\item Suppose $f(\omega)>\sup_n A(\omega\upharpoonright n)$. Choose rational $q_m$ such that $f(\omega)>q_m>\sup_n A(\omega\upharpoonright n)$ and $q_m<\sup_n A(\omega\upharpoonright n)+\epsilon$. Since $A$ is increasing, choose $n_0\geq 0$ such that $q_m<A(\omega\upharpoonright n)+\epsilon$ for all $n\geq n_0$. Further, since $f(\omega)>q_m$, there is $i\geq 0$ such that $\omega$ is in $[\tau_{m,i}]$. Let $n=\max(n_0,m,i,\left|\tau_{m,i}\right|)$ and let $\sigma=\omega\upharpoonright n$. Then $q_m<A(\sigma)+\epsilon$ and $m\leq \left|\sigma\right|$ and $i\leq \left|\sigma\right|$ and $\sigma \succeq \tau_{m,i}$. Then $A(\omega\upharpoonright (n+1))\geq q_m$, a contradiction.
\end{enumerate}

Finally, the equality $\sup_n \mathbb{E}_{\nu} A_n =\mathbb{E}_{\nu} f$ follows from the monotone convergence theorem.
\end{proof}

The following proposition is more complicated, and it involves constructing a computable probability measure $\mu$ from both the original computable probability measure $\nu$ and a computable dyadic $A$ which satisfies the properties from the previous proposition.
\begin{prop}\label{prop:iamcomplicated2}
Let $A:2^{<\mathbb{N}}\rightarrow \mathbb{Q}^{\geq 0}$ be increasing and computable, and such that $A(\emptyset)=0$ and $A(\sigma 0)=A(\sigma 1)$ for all $\sigma$ in $2^{<\mathbb{N}}$. Then, for all computable~$\nu$, there is a computable~$\mu$ such that $A$ is the predictable process of the dyadic $\nu$-submartingale  $L(\sigma)=-\ln \frac{\mu(\sigma)}{\nu(\sigma)}$.
\end{prop}
\begin{proof}
We define $L(\sigma)$ by recursion on length of $\sigma$, and set $\mu(\sigma)=e^{-L(\sigma)}\cdot \nu(\sigma)$ along the way, so that we have the equation $L(\sigma)=-\ln \frac{\mu(\sigma)}{\nu(\sigma)}$.

We define $L(\emptyset)=0$. Supposing that $L(\sigma)$ has been defined, we must define $L(\sigma 0)$ and $L(\sigma 1)$ such that both of the following hold:
\begin{align*}
& L(\sigma 0)\nu(\sigma 0\mid \sigma) +L(\sigma 1)\nu(\sigma 1\mid \sigma) =L(\sigma)+A(\sigma \iota)-A(\sigma) \\
& e^{-L(\sigma 0)} \nu(\sigma 0\mid \sigma)+e^{-L(\sigma 1)} \nu(\sigma 1\mid \sigma)= e^{-L(\sigma)}
\end{align*}
In this, the first equation comes from Theorem~\ref{thm:doobdecomp}(\ref{eqn:apartofdoob}), and the second equation is just the additivity condition for $\mu$ to be a measure, written out in terms of $L$ and $\nu$. To simplify the equations, we set $x:=L(\sigma 0)$ and $y:=L(\sigma 1)$ and we set $p:=\nu(\sigma 0\mid \sigma)$ and we set $c:=L(\sigma)+A(\sigma \iota)-A(\sigma)$ and $d:=e^{-L(\sigma)}$. Hence, given $p,c,d$ as described, we must solve for $x,y$ in the two following equations:
\begin{equation*}
xp +y(1-p)  = c, \hspace{10mm} e^{-x} p+e^{-y} (1-p) =d
\end{equation*}
This is the same as the equations $(e^{-x})^p \cdot (e^{-y})^{(1-p)} = e^{-c}$ and $e^{-x} p+e^{-y} (1-p) =d$. By setting $u:=e^{-x}$ and $v:=e^{-y}$, this means that for $p,c,d$ as described above, we must solve for $u,v>0$ in the pair of equations $u^p\cdot v^{1-p} = e^{-c}$ and $p\cdot u +(1-p)\cdot v  =d$. We can rewrite these by solving for $u$ in each:
\begin{equation*}
u = \frac{e^{-\frac{c}{p}}}{v^{\frac{1-p}{p}}}, \hspace{10mm}
u = \frac{d-(1-p)\cdot v}{p}  
\end{equation*}
Note that the first of these implies that $u>0$ when $v>0$. By setting these equal to each other and cross multiplying, it remains to solve for $v>0$ in the equation $p \cdot e^{-\frac{c}{p}}  = d\cdot v^{\frac{1-p}{p}} -(1-p)\cdot v^{\frac{1}{p}}$. We simplify as: $(1-p)\cdot v^{\frac{1}{p}}-d\cdot v^{\frac{1}{p}-1}+p\cdot e^{-\frac{c}{p}}=0$. Consider the real-valued computable continuous function $f:\mathbb{R}^{\geq 0}\rightarrow \mathbb{R}$ given by 
\begin{equation*}
f(v)=(1-p)\cdot v^{\frac{1}{p}}-d\cdot v^{\frac{1}{p}-1}+p\cdot e^{-\frac{c}{p}}
\end{equation*}
Note that $f(0)=p\cdot e^{-\frac{c}{p}}>0$. Further recall that $d>0$ and note that one has
\begin{equation*}
f(d) = (1-p)\cdot d^{\frac{1}{p}}-d^{\frac{1}{p}}+p\cdot e^{-\frac{c}{p}} = p\cdot (e^{-\frac{c}{p}}-d^{\frac{1}{p}})
\end{equation*}
Further, recalling the definitions of $c,d$ from above we have:
\begin{equation*}
f(d) = p\cdot ((e^{-\frac{1}{p}})^{L(\sigma)+(A(\sigma \iota)-A(\sigma))}-(e^{-\frac{1}{p}})^{L(\sigma)})
\end{equation*}
Since $0<e^{-\frac{1}{p}}<1$ and since for fixed base $0<a<1$ one has that $g(x)=a^x$ is decreasing and since $L(\sigma)\leq L(\sigma)+(A(\sigma \iota)-A(\sigma))$ due to $A$ being increasing, one has that $f(d)\leq 0$. Further, since $A$ is rational-valued and computable, we can computably test for the whether the condition $A(\sigma \iota)=A(\sigma)$ is satisfied. If this condition is satisfied, then $f(d)=0$. If this condition is not satisfied, then since $g(x)=a^x$ is strictly decreasing for fixed base $0<a<1$, we have that $f(d)<0$. In this case, since $f(d)< 0<f(0)$ and $f$ is strictly decreasing on $(0,d)$, by the effective Intermediate Value Theorem we can find a computable real number $v$ in the interval $(0,d)$ such that $f(v)=0$.\footnote{\label{fn:effivt} Cf. \cite[Theorem II.6.6]{Simpson2009aa}. In general, the effective Intermediate Value Theorem is non-uniform, but one of the special cases where it is uniform is when the solution is unique (cf. \cite[\S{10.4} pp. 371-373 and Exercise 10.9.14(2) p. 398]{Dzhafarov2022-bx}).} In either case, we can then find a computable real number $v$ in the interval $(0,d]$ such that $f(v)=0$; and we are done since the case-break is effective.
\end{proof}

\begin{proofdetail}
To see that $f$ is strictly decreasing on $(0,d)$, it suffices to show that its derivative is negative on $(0,d)$. Its derivative is:
\begin{equation*}
f^{\prime}(v)=(\frac{1}{p}-1)\cdot v^{\frac{1}{p}-1} - d\cdot (\frac{1}{p}-1) \cdot v^{\frac{1}{p}-2}
\end{equation*}
Then for $v>0$ one has: $f^{\prime}(v)<0$ iff $(\frac{1}{p}-1)\cdot v^{\frac{1}{p}-2} \cdot (v-d)<0$ iff $v<d$. Hence $f^{\prime}$ is negative on $(0,d)$.
\end{proofdetail}

We are now in a position to prove our main result. In this proof, we use the well-known characterization, due to Levin (\cite{Levin1976}), that $\omega$ is in $\mathsf{MLR}^{\nu}$ iff $f(\omega)<\infty$ for all lsc $f:2^{\mathbb{N}}\rightarrow [0,\infty]$ with finite $\nu$-expectation. These lsc functions are called \emph{Martin-L\"of $L_1(\nu)$ tests}. Similarly,  $\omega$ is in $\mathsf{SR}^{\nu}$ iff $f(\omega)<\infty$ for all lsc $f:2^{\mathbb{N}}\rightarrow [0,\infty]$ with finite computable $\nu$-expectation. These lsc functions are called \emph{Schnorr $L_1(\nu)$ tests}, and this characterization is due to Miyabe (\cite[Theorem 3.5]{Miyabe2013}).

\thmmainmlrweak*

\begin{proof}
Suppose that $\omega$ is in $\mathsf{MLR}^{\nu}$. Suppose that $\nu\abcon_{kl} \mu$. Let $L(\sigma)=-\ln \frac{\mu(\sigma)}{\nu(\sigma)}$ be the dyadic $\nu$-submartingale. Then $\sup_n \mathbb{E}_{\nu} L_n<\infty$. Let $L=N+A$ be the Doob decomposition. Since $N(\emptyset)=L(\emptyset)=0$, we have $\mathbb{E}_{\nu} A_n=\mathbb{E}_{\nu} L_n- \mathbb{E}_{\nu} N_n = \mathbb{E}_{\nu} L_n- \mathbb{E}_{\nu} N_0 =\mathbb{E}_{\nu} L_n$. Then $\sup_n \mathbb{E}_{\nu} A_n=\sup_n \mathbb{E}_{\nu} L_n<\infty$. Let $f(\omega)=\lim_n A_n(\omega) = \sup_n A(\omega\upharpoonright n)$. By the Monotone Convergence Theorem, we have $\mathbb{E}_{\nu} f =\lim_n \mathbb{E}_{\nu} A_n=\sup_n \mathbb{E}_{\nu} A_n<\infty$. Then $f$ is a Martin-L\"of $L_1(\nu)$ test. Since $\omega$ is in $\mathsf{MLR}^{\nu}$, we have $f(\omega)<\infty$. Further, since $A_n$ is non-negative and increasing and $A(\emptyset)=L(\emptyset)-N(\emptyset)=0$, by telescoping we have $f(\omega)=\sum_n (A_{n+1}(\omega)-A_n(\omega))$. Then by Theorem~\ref{prop:klandtheincrements} we have 
\begin{equation*}
\sum_n D_{\mathscr{F}_{n+1}}(\nu \mid \mu)(\omega) =\sum_n \big(A(\omega\upharpoonright (n+1)\,)-A(\omega\upharpoonright n\,)\big) = f(\omega)<\infty
\end{equation*}

Conversely, suppose $\omega$ is in  $\mathsf{MR}^{\nu}_1(\abcon_{kl}, \mathscr{F}_{n+1},D)$; we must show $\omega$ is in $\mathsf{MLR}^{\nu}$. Suppose that $f:2^{\mathbb{N}}\rightarrow [0,\infty]$ is a Martin-L\"of $L_1(\nu)$ test; we must show that $f(\omega)<\infty$. By Proposition~\ref{prop:frommlrtest2athing}, there is increasing computable $A:2^{<\mathbb{N}}\rightarrow \mathbb{Q}^{\geq 0}$ such that 
\begin{enumerate}[leftmargin=*]
    \item\label{prop:awesome:1v2} $A(\emptyset)=0$, and $A(\sigma 0)=A(\sigma 1)$ for all $\sigma$ in $2^{<\mathbb{N}}$
    \item\label{prop:awesome:2v2} $f(\omega)=\sup_n A(\omega\upharpoonright n)$
    \item\label{prop:awesome:4v2} $\sup_n \mathbb{E}_{\nu} A_n =\mathbb{E}_{\nu} f$.
\end{enumerate}
From $f$ being a Martin-L\"of $L_1(\nu)$ test and from (\ref{prop:awesome:4v2}), we further have:
\begin{enumerate}[leftmargin=*] \setcounter{enumi}{3}
    \item\label{prop:awesome:4v3} $\sup_n \mathbb{E}_{\nu} A_n <\infty$.
\end{enumerate}
By (\ref{prop:awesome:1v2}), we can apply Proposition~\ref{prop:iamcomplicated2} to obtain a computable $\mu$ such that $A$ is the predictable process of $L$, where 
$L(\sigma)=-\ln \frac{\mu(\sigma)}{\nu(\sigma)}$.
Then the Doob decomposition of $L$ is $L=N+A$. Since $L(\emptyset)=0$, we have $N(\emptyset)=0$ and hence:
\begin{equation}\label{eqn:refertome}
\mathbb{E}_{\nu} L_n=\mathbb{E}_{\nu} L_n- \mathbb{E}_{\nu} N_0=\mathbb{E}_{\nu} L_n- \mathbb{E}_{\nu} N_n=\mathbb{E}_{\nu} A_n
\end{equation}
From this and (\ref{prop:awesome:4v3}) we get $\nu\abcon_{kl} \mu$. Since $\omega$ is in $\mathsf{MR}^{\nu}_1(\abcon_{kl}, \mathscr{F}_{n+1},D)$, we have $\sum_n D_{\mathscr{F}_{n+1}}(\nu \mid \mu)(\omega)<\infty$. Further, since $A_n$ is non-negative and increasing and $A(\emptyset)=0$, by (\ref{prop:awesome:2v2}) and telescoping  we have $f(\omega)=\sum_n (A_{n+1}(\omega)-A_n(\omega))$. By Theorem~\ref{prop:klandtheincrements}, we have 
\begin{equation*}
 f(\omega) = \sum_n \big(A(\omega\upharpoonright (n+1)\,)-A(\omega\upharpoonright n\,)\big) = \sum_n D_{\mathscr{F}_{n+1}}(\nu \mid \mu)(\omega)<\infty
 \end{equation*}

The proof for Schnorr randomness is the same. In the forward direction, the hypothesis that $\nu\abcon_{klc} \mu$ implies that $\sup_n \mathbb{E}_{\nu} A_n$ is finite and computable, and hence that $f$ is a Schnorr $L_1(\nu)$ test. In the backward direction, the hypothesis that  $f$ is a Schnorr $L_1(\nu)$ test implies and (\ref{prop:awesome:4v2}) and (\ref{eqn:refertome}) implies that $\nu\abcon_{klc} \mu$ in addition to $\nu\abcon_{kl} \mu$.
\end{proof}

\section{Hellinger distance and proof of Corollary~\ref{corvovkish}}\label{sec:proof:hellinger}

First, we note the following formulas for the Hellinger affinity and squared Hellinger distance in the weak setting:
\begin{prop}\label{prop:formula:hellinger}
\begin{align*}
\alpha_{\mathscr{F}_{n+1}}(\nu,\mu)(\omega)& = \sum_{\iota\in \{0,1\}} \sqrt{\nu(\omega\upharpoonright n\, \iota \mid \omega\upharpoonright n) \cdot \mu(\omega\upharpoonright n\, \iota \mid \omega\upharpoonright n)} \\
H^2_{\mathscr{F}_{n+1}}(\nu,\mu)(\omega) & = 2\big( 1- \sum_{\iota\in \{0,1\}} \sqrt{\nu(\omega\upharpoonright n\, \iota \mid \omega\upharpoonright n) \cdot \mu(\omega\upharpoonright n\, \iota \mid \omega\upharpoonright n)}\big) 
\end{align*}
\end{prop}
\begin{proof}
By definition of Hellinger affinity (cf. Definition~\ref{defn:threedistance}(\ref{defn:threedistance:2})) and by  (\ref{eqn:thedervweak}):
\begin{align*}
 \alpha_{\mathscr{F}_{n+1}}(\nu,\mu)(\omega) & = \int \sqrt{\frac{d(\nu(\cdot \mid \mathscr{F}_n)(\omega) \upharpoonright \mathscr{F}_{n+1})}{d(\mu(\cdot \mid \mathscr{F}_n)(\omega) \upharpoonright \mathscr{F}_{n+1})}(\omega^{\prime})} \; d(\mu(\cdot \mid \mathscr{F}_n)(\omega))(\omega^{\prime}) \\
=&  \frac{1}{\mu([\omega\upharpoonright n])} \int_{[\omega\upharpoonright n]} \sqrt{\frac{d(\nu(\cdot \mid \mathscr{F}_n)(\omega) \upharpoonright \mathscr{F}_{n+1})}{d(\mu(\cdot \mid \mathscr{F}_n)(\omega) \upharpoonright \mathscr{F}_{n+1})}(\omega^{\prime})} \; d\mu(\omega^{\prime}) \\
=&   \frac{1}{\mu([\omega\upharpoonright n])} \cdot \bigg (\sum_{\iota\in \{0,1\}} \sqrt{\frac{\nu(\omega\upharpoonright n\, \iota \mid \omega\upharpoonright n)}{\mu(\omega\upharpoonright n\, \iota \mid \omega\upharpoonright n)}} \cdot \mu([\omega\upharpoonright n\, \iota])\bigg) \\ 
=&  \sum_{\iota\in \{0,1\}} \sqrt{\nu(\omega\upharpoonright n\, \iota \mid \omega\upharpoonright n) \cdot \mu(\omega\upharpoonright n\, \iota \mid \omega\upharpoonright n)}
\end{align*}
Then we are done since the Hellinger affinity $\alpha$ and Hellinger distance $H$ are related by the equation $2(1-\alpha) = H^2$.
\end{proof}

Second we show the consequence of the Kabanov-Lipcer-Shiryaev results  mentioned in \S\ref{sec:intro} (see \cite{Kabanov1977-yr} and \cite[Theorem 4, pp. 169-171]{Shiryaev2019-rh}). This is a classical result, and no assumptions are made about the computability of the measures.
\thmshiryaev*
\begin{proof}
If $\mathscr{G}_n$ is an increasing sequence of sub-$\sigma$-algebras of the Borel $\sigma$-algebra whose union generates the Borel $\sigma$-algebra, then we say that $\nu$ is \emph{locally absolutely continuous} with respect to $\mu$, abbreviated $\nu\abcon^{loc} \mu$, if $\nu\upharpoonright \mathscr{G}_n\abcon \mu\upharpoonright \mathscr{G}_n$ for all $n\geq 0$. Assuming $\nu\abcon^{loc} \mu$, Kabanov-Lipcer-Shiryaev show that $\nu\abcon \mu$ iff for $\nu$-a.s. many $\omega$ one has: 
\begin{equation}
\sum_n \bigg(1-\mathbb{E}_{\mu}\left[\sqrt{\frac{d(\nu \upharpoonright \mathscr{G}_{n+1})}{d(\mu \upharpoonright \mathscr{G}_{n+1})}\Big/\frac{d(\nu \upharpoonright \mathscr{G}_n)}{d(\mu \upharpoonright \mathscr{G}_n)}} \, \Big| \, \mathscr{G}_n\right](\omega)\bigg)<\infty 
\end{equation}

Since we are restricting to measures $\mu,\nu$ on Cantor space which have full support, we trivially have $\nu\upharpoonright \mathscr{F}_n\abcon \mu\upharpoonright \mathscr{F}_n$ for all $n\geq 0$, where again $\mathscr{F}_n$ is the $\sigma$-algebra generated by the basic clopens associated to the length $n$-strings. Hence, the Kabanov-Lipcer-Shiryaev theorem then implies that $\nu\abcon \mu$ iff for $\nu$-a.s. many $\omega$ one has: 
\begin{equation}\label{eqn:bigshiryaevthing}
\sum_n \bigg(1-\mathbb{E}_{\mu}\left[\sqrt{\frac{d(\nu \upharpoonright \mathscr{F}_{n+1})}{d(\mu \upharpoonright \mathscr{F}_{n+1})} \Big/ \frac{d(\nu \upharpoonright \mathscr{F}_n)}{d(\mu \upharpoonright \mathscr{F}_n)}} \, \Big| \, \mathscr{F}_n \right](\omega)\bigg)<\infty  
\end{equation}
Hence, it suffices to show that this sum is equal to $\sum_n \frac{1}{2} H^2_{\mathscr{F}_{n+1}}(\mu,\nu)(\omega)$, since the presence of the multiplicative constant $\frac{1}{2}$ does not affect convergence of the sum.

For any $k\geq 0$ one has $\big(\nicefrac{d(\nu \upharpoonright \mathscr{F}_k)}{d(\mu \upharpoonright \mathscr{F}_k)}\big)(\omega^{\prime}) = \frac{\nu(\omega^{\prime}\upharpoonright k)}{\mu(\omega^{\prime}\upharpoonright k)}$. Hence for any $n\geq 0$:
\begin{equation}
\left(\frac{d(\nu \upharpoonright \mathscr{F}_{n+1})}{d(\mu \upharpoonright \mathscr{F}_{n+1})} \Big/ \frac{d(\nu \upharpoonright \mathscr{F}_n)}{d(\mu \upharpoonright \mathscr{F}_n)}\right)(\omega^{\prime})=\frac{\nu(\omega^{\prime}\upharpoonright (n+1))}{\mu(\omega^{\prime}\upharpoonright (n+1))} \Big/ \frac{\nu(\omega^{\prime}\upharpoonright n)}{\mu(\omega^{\prime}\upharpoonright n)} = \frac{\nu(\omega^{\prime}\upharpoonright (n+1) \mid \omega^{\prime}\upharpoonright n)}{\mu(\omega^{\prime}\upharpoonright (n+1) \mid \omega^{\prime}\upharpoonright n)} 
\end{equation}
Then one has the following expression for the conditional expectation in (\ref{eqn:bigshiryaevthing}):
\begin{align}
& \mathbb{E}_{\mu}\left[\sqrt{\frac{d(\nu \upharpoonright \mathscr{F}_{n+1})}{d(\mu \upharpoonright \mathscr{F}_{n+1})} \Big/ \frac{d(\nu \upharpoonright \mathscr{F}_n)}{d(\mu \upharpoonright \mathscr{F}_n)}} \, \Big| \mathscr{F}_n \right](\omega) \notag \\
= & \frac{1}{\mu([\omega\upharpoonright n])} \int_{[\omega\upharpoonright n]} \sqrt{\frac{\nu(\omega^{\prime}\upharpoonright (n+1) \mid \omega^{\prime}\upharpoonright n)}{\mu(\omega^{\prime}\upharpoonright (n+1) \mid \omega^{\prime}\upharpoonright n)}} \; d\mu(\omega^{\prime}) \notag \\
= & \sum_{\iota\in \{0,1\}}  \frac{1}{\mu([\omega\upharpoonright n])} \cdot \sqrt{\frac{\nu(\omega\upharpoonright n\, \iota \mid \omega\upharpoonright n)}{\mu(\omega\upharpoonright n\, \iota \mid \omega\upharpoonright n)}} \cdot \mu([\omega\upharpoonright n\, \iota])\notag  \\
= & \sum_{\iota\in \{0,1\}} \sqrt{\nu(\omega\upharpoonright n\, \iota \mid \omega\upharpoonright n) \cdot \mu(\omega\upharpoonright n\, \iota \mid \omega\upharpoonright n)} \notag \\
= &  \alpha_{\mathscr{F}_{n+1}}(\nu,\mu)(\omega) \notag
\end{align}
where the last line follows from the previous proposition.  Then we are done since, again, the Hellinger distance $H$ and the Hellinger affinity $\alpha$ are related by the equation $1-\alpha = \frac{1}{2} H^2$.
\end{proof}

At the end of his paper \cite{Vovk1987}, Vovk cites the Kabanov-Lipcer-Shiryaev paper \cite{Kabanov1977-yr}, and its later presentation in Shiryaev's book (the most recent edition being \cite[Theorem 4 pp. 169-171]{Shiryaev2019-rh}). Vovk writes: ``A related result\textemdash a criterion of absolute continuity and singularity of probability measures in `predictable' terms\textemdash has been obtained in probability theory'' (\cite[p. 5]{Vovk1987}). In addition to being useful, Theorem~\ref{thmshiryaev} is our attempt to distill the connection between Kabanov-Lipcer-Shiryaev's work and Vovk's Theorem~\ref{Vovk}.

Using Vovk's Theorem~\ref{Vovk} and Theorem~\ref{thmshiryaev}, we can verify the following component of Diagram~\ref{eqn:abcondiagram}
\begin{prop}\label{prop:shiryaev}
If  $\nu\abcon_{\mathsf{MLR}} \mu$ then $\nu\abcon \mu$.
\end{prop}

Note that this also follows from the proof of \cite[Proposition 39(a)]{Bienvenu2009-jf}. That proposition is stated in terms of mutual absolute continuity, but the proof gives the result for one-sided absolute continuity.

\begin{proof}
Suppose that $\nu\abcon_{\mathsf{MLR}} \mu$. Then by the forward direction of Vovk's Theorem~\ref{Vovk}, every $\omega$ in the $\nu$-measure one set $\mathsf{MLR}^{\nu}$ satisfies $\sum_n H^2_{\mathscr{F}_{n+1}}(\mu,\nu)(\omega)<\infty$. Then we are done by the backwards direction of Theorem~\ref{thmshiryaev}.
\end{proof}

Now we turn towards finishing Corollary~\ref{corvovkish}. We first prove a preliminary proposition:
\begin{prop}\label{prop:easyconstruction}
Suppose that $f:2^{\mathbb{N}}\rightarrow [0,\infty]$ is a Martin-L\"of $L_1(\nu)$ test. Then there is a computable $\mu$ such that for all $\omega$ in $2^{\mathbb{N}}$ one has $f(\omega)=\sum_n H^2_{\mathscr{F}_{n+1}}(\mu,\nu)(\omega)$.
\end{prop}
\begin{proof}
By Proposition~\ref{prop:frommlrtest2athing}, there is a computable increasing dyadic function $A:2^{<\mathbb{N}}\rightarrow \mathbb{Q}^{\geq 0}$ such that $f(\omega)=\sup_n A_n(\omega)$ and $A(\emptyset)=0$ and $A(\sigma 0)=A(\sigma 1)$ for all $\sigma$ in $2^{<\mathbb{N}}$ and further $A(\sigma\iota)-A(\sigma)<\epsilon$ for all $\sigma$ in $2^{<\mathbb{N}}$ and $\iota$ in $\{0,1\}$, where we fix $\epsilon:=2(1-\sqrt{\nicefrac{1}{2}})$. We define $\mu$ by defining $\mu(\sigma)$ recursively on length of $\sigma$ so that it satisfies the following for each $\iota$ in $\{0,1\}$:
\begin{equation}\label{prop:whatwewantasigmaiota}
A(\sigma\iota)-A(\sigma)=2(1-\sum_{\jmath \in \{0,1\}} \sqrt{\mu(\sigma\jmath \mid \sigma)\nu(\sigma\jmath \mid \sigma)}\,)
\end{equation}
Let $\sigma,\iota$ be fixed and suppose we have already defined $\mu(\sigma)$. Define rational
\begin{equation*}
c=1- \frac{1}{2}(A(\sigma\iota)-A(\sigma))
\end{equation*}
Note by our choice of $\epsilon>0$, we have that $c$ is in the interval $(\sqrt{\nicefrac{1}{2}},1]$. Then either $\sqrt{\nu(\sigma 0\mid \sigma)}<c$ or $\sqrt{\nu(\sigma 1\mid \sigma)}<c$ or both. Since these are c.e. relations, enumerate them until one finds $\jmath$ in $\{0,1\}$ with $\sqrt{\nu(\sigma \jmath\mid \sigma)}<c$. Define computable real $p=\nu(\sigma \jmath\mid \sigma)$ so that $\sqrt{p}<c$. Consider the computable continuous $g:[p,1]\rightarrow \mathbb{R}$ defined by
\begin{equation*}
g(x)=\sqrt{xp}+\sqrt{(1-x)(1-p)}
\end{equation*}
One has $g(p)=1\geq c>\sqrt{p}=g(1)$. Further, since $c$ is rational, we can compute whether $c=1$. If so, then we choose $x=p$, so that we have $g(x)=c$. If not, then since $g(p)>c>g(1)$ and since $g$ is strictly decreasing on $(p,1)$, by the effective Intermediate Value Theorem (cf. footnote~\ref{fn:effivt}), we can compute $x$ in the interval $(p,1)$ such that $g(x)=c$. Hence in either case, we can compute $x$ in the interval $[p,1)$ such that $g(x)=c$; and the case break is effective. Then we define $\mu(\sigma\jmath \mid \sigma)=x$ and $\mu(\sigma(1-\jmath) \mid \sigma)=1-x$. This finishes the construction of $\mu$ such that (\ref{prop:whatwewantasigmaiota}) holds. Then by $A(\emptyset)=0$ and $A$ being non-negative, we can telescope to obtain 
\begin{equation*}
f(\omega) = \sum_n (A_{n+1}(\omega)-A_n(\omega)) = \sum_n H^2_{\mathscr{F}_{n+1}}(\mu,\nu)(\omega).
\end{equation*}
where the last equation follows from (\ref{prop:whatwewantasigmaiota}) and Proposition~\ref{prop:formula:hellinger}. 
\end{proof}

\begin{proofdetail}
To see that $g$ is strictly decreasing on $(p,1)$, it suffices to show that its derivative is negative on $(p,1)$. Its derivative is:
\begin{equation*}
g^{\prime}(x)=\frac{1}{2} \frac{\sqrt{p}}{\sqrt{x}} - \frac{1}{2} \frac{\sqrt{1-p}}{\sqrt{1-x}}
\end{equation*}
For $0<x<1$ one has: $g^{\prime}(x)<0$ iff $ \frac{\sqrt{p}}{\sqrt{x}}<\frac{\sqrt{1-p}}{\sqrt{1-x}}$ iff $\sqrt{p}\sqrt{1-x}<\sqrt{1-p}\sqrt{x}$ iff $p(1-x)<(1-p)x$ iff $p-px<x-px$ iff $p<x$. Hence indeed $g^{\prime}(x)<0$ for $x$ in $(p,1)$.
\end{proofdetail}

\corvovkish*
\begin{proof}
The forward inclusion follows from the forward direction of Vovk's Theorem~\ref{Vovk}. For the backward inclusion, suppose that $\omega$ is in $\mathsf{MR}^{\nu}_2(\abcon_\mathsf{MLR}, \mathscr{F}_{n+1},H)$. To show that $\omega$ is in $\mathsf{MLR}^{\nu}$, let $f$ be an $L_1(\nu)$ Martin-L\"of test; we must show that $f(\omega)<\infty$. Let $\mu$ be as in Proposition~\ref{prop:easyconstruction}. Then we note that every $\omega^{\prime}$ in $\mathsf{MLR}^{\nu}$ is such that $f(\omega^{\prime})<\infty$ and hence $\sum_n H^2_{\mathscr{F}_{n+1}}(\mu,\nu)(\omega^{\prime})<\infty$, and hence by the backwards direction of Vovk's Theorem, $\omega^{\prime}$ is in $\mathsf{MLR}^{\mu}$. Then $\nu\abcon_{\mathsf{MLR}} \mu$ and since $\omega$ is in $\mathsf{MR}^{\nu}_2(\abcon_\mathsf{MLR}, \mathscr{F}_{n+1},H)$, we have that $\sum_n H^2_{\mathscr{F}_{n+1}}(\mu,\nu)(\omega)<\infty$, and so $f(\omega)<\infty$.
\end{proof}

\section{Total variational distance}\label{sec:totalvariational}

Now we shift attention to the total variational distance.  The following notion features in the statement of our main result for this setting, Theorem~\ref{thm:crzero}.

\begin{defn}\label{defn:mild}
A point $\omega$ is $\nu$-\emph{mild} if $\liminf_n \nu(\omega\upharpoonright (n+1)\mid \omega\upharpoonright n)>0$, and we abbreviate this as $\mathsf{Mild}^{\nu}$.  
\end{defn}

Note that, since we are restricting attention to probability measures $\nu$ with full support, one has that $0<\nu(\omega\upharpoonright (n+1)\mid \omega\upharpoonright n)<1$. Hence, this definition is equivalent to one with the liminf replaced by inf.

For Bernoulli measures $\nu$ or strongly positive Bernoulli measures $\nu$ \cite[Definition 20]{Bienvenu2009-jf} all sequences are $\nu$-mild. In the general case, mildness is not universally a measure-one property.\footnote{For example, let $X_1, X_2, \hdots$ be a sequence random variables with $X_n:2^{\mathbb{N}} \to \{0,1\}$ and $X_n(\omega) = 1$ if $\omega(n) =1$ and $X_n(\omega) = 0$ otherwise. Suppose that the $X_n$ are independently distributed with $\mu(X_n) = \frac{1}{n+1}$. Let $\mu^\infty$ be the infinite product measure $(\mu\circ X_1^{-1}) \times (\mu\circ X_2^{-1}) \times \cdots$, which is computable since $\mu\circ X_n^{-1}$ is uniformly computable. Furthermore, $\mu^\infty$ has full support. Then, by the second Borel-Cantelli lemma, $X_n = 1$ infinitely often with $\mu^\infty$ probability one. But any sample $\omega$ which satisfies $X_n(\omega)=1$ infinitely often cannot be $\mu^{\infty}$-mild since $\mu^\infty(X_{n+1} = 1 \mid X_1, \hdots, X_n) = \mu(X_{n+1} = 1) \to 0$ as $n \to \infty$. We thank an anonymous reviewer for pointing this out as a counterexample to a result about mildness we had previously asserted.}

In the Cantor space setting, one has the following formula for the weak total variational distance:
\begin{prop}\label{prop:equiv:char:bdweak1}
Suppose $\mu,\nu,\omega$ are given. Then
\begin{equation}\label{eqn:tvweakformulacantor}
T_{\mathscr{F}_{n+1}}(\mu,\nu)(\omega)=\left| \mu( \omega \upharpoonright (n+1) \mid  \omega \upharpoonright n) - \nu( \omega \upharpoonright (n+1) \mid  \omega \upharpoonright n)\right|
\end{equation}
\end{prop}
\begin{proof}
For an event $A$ in $\mathscr{F}_{n+1}$ and for a string $\sigma$ with $\left|\sigma\right|=n$, one has that $A\cap  [\sigma]$ is either empty or equal to~$[\sigma]$ or equal to~$[\sigma \iota]$ for some $\iota\in \{0,1\}$. Letting $\sigma =\omega\upharpoonright n$, one has that in the first two cases the quantity $\left|\mu(A\mid \mathscr{F}_n)(\omega)  - \nu(A\mid \mathscr{F}_n)(\omega)\right|$ is zero, while in the third case this quantity is equal to the right-hand side of (\ref{eqn:tvweakformulacantor}).
\end{proof}

\begin{prop}\label{prop:equivalentliminf}
The following are equivalent for $\omega$ in $\mathsf{Mild}^{\nu}$:
\begin{enumerate}[leftmargin=*]
    \item \label{prop:equiv:char:bdweak1:1v2} $\lim_n T_{\mathscr{F}_{n+1}}(\mu,\nu)(\omega)=0$
    \item \label{prop:equiv:char:bdweak1:2v2} $\lim_n \left| \mu( \omega \upharpoonright (n+1) \mid  \omega \upharpoonright n) - \nu( \omega \upharpoonright (n+1) \mid  \omega \upharpoonright n)\right|=0$
    \item \label{prop:equiv:char:bdweak1:3v2} $\lim_n \frac{M_{n+1}(\omega)}{M_n(\omega)}=1$   
\end{enumerate}
where $M(\sigma)=\frac{\mu(\sigma)}{\nu(\sigma)}$ is the computable dyadic $\nu$-martingale. 
\end{prop}
\begin{proof}
The equivalence of (\ref{prop:equiv:char:bdweak1:1v2})-(\ref{prop:equiv:char:bdweak1:2v2}) follows from Proposition~\ref{prop:equiv:char:bdweak1}.

Suppose (\ref{prop:equiv:char:bdweak1:2v2}); we show (\ref{prop:equiv:char:bdweak1:3v2}). Suppose $\omega$ is $\mathsf{Mild}^{\nu}$, so that $c>0$, where $c:=\liminf_n \nu( \omega \upharpoonright (n+1) \mid  \omega \upharpoonright n)$. Let $\epsilon>0$.  Choose $n_0\geq 0$ such that for all $n_1\geq n_0$ one has $\inf_{n\geq n_1} \nu( \omega \upharpoonright (n+1) \mid  \omega \upharpoonright n)>\frac{c}{2}$. So for all $n_1\geq n_0$ and all $n\geq n_1$ one has the bound $\frac{1}{\nu( \omega \upharpoonright (n+1) \mid  \omega \upharpoonright n)}<\frac{2}{c}$. From (\ref{prop:equiv:char:bdweak1:2v2}) applied to $\epsilon\cdot \frac{c}{2}$, choose $n_1\geq n_0$ such that for all $n\geq n_1$ one has that  $\left| \mu( \omega \upharpoonright (n+1) \mid  \omega \upharpoonright n) - \nu( \omega \upharpoonright (n+1) \mid  \omega \upharpoonright n)\right|<\epsilon \cdot \frac{c}{2}$. Let $n\geq n_1$. Then by dividing by the  quantity $0<\nu( \omega \upharpoonright (n+1) \mid  \omega \upharpoonright n)<1$ and using the previous bound, we have
$\left|\frac{ \mu( \omega \upharpoonright (n+1) \mid  \omega \upharpoonright n)}{ \nu( \omega \upharpoonright (n+1) \mid  \omega \upharpoonright n)}-1\right|< \epsilon\cdot \frac{c}{2} \cdot \frac{1}{\nu( \omega \upharpoonright (n+1) \mid  \omega \upharpoonright n)}<\epsilon$. Then using (\ref{rmk:martingalebd:eqn:redeux}) we are done. 

Suppose (\ref{prop:equiv:char:bdweak1:3v2}); we show (\ref{prop:equiv:char:bdweak1:2v2}). Let $\epsilon>0$. From $\lim_n \frac{M(\omega\upharpoonright (n+1))}{M(\omega\upharpoonright n)}=1$, choose $n_0\geq 0$ such that for all $n\geq n_0$ one has $\left|\frac{M(\omega\upharpoonright (n+1))}{M(\omega\upharpoonright n)}-1\right|<\epsilon$. By (\ref{rmk:martingalebd:eqn:redeux}), one has $\left|\frac{ \mu( \omega \upharpoonright (n+1) \mid  \omega \upharpoonright n)}{ \nu( \omega \upharpoonright (n+1) \mid  \omega \upharpoonright n)}-1\right|<\epsilon$. By multiplying by the quantity $0<\nu( \omega \upharpoonright (n+1) \mid  \omega \upharpoonright n)<1$, we have $\left| \mu( \omega \upharpoonright (n+1) \mid  \omega \upharpoonright n)- \nu( \omega \upharpoonright (n+1) \mid  \omega \upharpoonright n)\right|<\epsilon\cdot \nu( \omega \upharpoonright (n+1) \mid  \omega \upharpoonright n)< \epsilon$.
\end{proof}

\section{Computable randomness}\label{sec:cr}

Another traditional randomness notion, also due to Schnorr \cite{Schnorr1971aa}, is computable randomness. A point $\omega$ is \emph{computable $\nu$-random}, abbreviated $\mathsf{CR}^{\nu}$, if $\sup_n N(\omega\upharpoonright n)<\infty$ for all  computable non-negative dyadic $\nu$-martingales $N$. By formalizing the Upcrossing Lemma, one can show that $\omega$ in $\mathsf{CR}^{\nu}$ iff $\lim_n N(\omega\upharpoonright n)$ exists and is finite for all computable non-negative dyadic $\nu$-martingales~$N$ (the proof in \cite[Theorem 7.1.3]{Downey2010aa}, which follows \cite{Schn:1971}, generalizes to arbitrary computable probability measures on Cantor space). 

The following analogue of Vovk's Theorem~\ref{Vovk} for computable randomness is elementary but illustrative:
\begin{prop}\label{crabcon}
The following are equivalent for $\omega$ in $\mathsf{CR}^{\nu}$:
\begin{enumerate}[leftmargin=*]
    \item \label{crabcon:1} $\omega$ is in $\mathsf{CR}^{\mu}$.
    \item \label{crabcon:2} $\lim_n \frac{\mu(\omega\upharpoonright n)}{\nu(\omega\upharpoonright n)}$ exists and is finite and non-zero.
    \item \label{crabcon:2.2} $\lim_n \frac{\nu(\omega\upharpoonright n)}{\mu(\omega\upharpoonright n)}$ exists and is finite and non-zero.
    \item \label{crabcon:3} $\sup_n \frac{\nu(\omega\upharpoonright n)}{\mu(\omega\upharpoonright n)}<\infty$.
\end{enumerate}
\end{prop}
\begin{proof}
First suppose (\ref{crabcon:1}); we show (\ref{crabcon:2}). Let $N=\frac{\mu}{\nu}$, which is a computable dyadic $\nu$-martingale. Likewise, $M=\frac{\nu}{\mu}$ is a computable dyadic $\mu$-martingale. Since $\omega$ is in both $\mathsf{CR}^{\nu}$ and $\mathsf{CR}^{\mu}$, one has that $\lim_n N(\omega\upharpoonright n)$ and $\lim_n M(\omega\upharpoonright n)$ exist. Since $M=N^{-1}$, these two limits must be non-zero.

One has that (\ref{crabcon:2}) and (\ref{crabcon:2.2}) are trivially equivalent. Likewise, (\ref{crabcon:2.2}) trivially implies~(\ref{crabcon:3}). 

Suppose (\ref{crabcon:3}); we show (\ref{crabcon:1}). Suppose that $M$ is a computable dyadic $\mu$-martingale. Let $N=M\cdot \frac{\mu}{\nu}$. Then it is easy to check that $N$ is a computable dyadic $\nu$-martingale. Further, $N$ is computable since $M, \mu, \nu$ are computable. Since $\omega$ in $\mathsf{CR}^{\nu}$, there is $K>0$ such that $N(\omega\upharpoonright n)<K$ for all $n\geq 0$. By (\ref{crabcon:3}), choose $K^{\prime}>0$ such that $ \frac{\nu(\omega\upharpoonright n)}{\mu(\omega\upharpoonright n)}<K^{\prime}$ for all $n\geq 0$. Then one has $M(\omega\upharpoonright n)=N(\omega\upharpoonright n)\cdot \frac{\nu(\omega\upharpoonright n)}{\mu(\omega\upharpoonright n)}\leq K\cdot K^{\prime}$ for all $n\geq 0$. \end{proof}

Computable randomness can also be characterized via sequential tests:
\begin{defn}\label{defn:sequentialcr}
A \emph{bounded sequential Martin-L\"of $\nu$-test} is given by a computable sequence $V_n$ of c.e. opens and a computable probability measure $\rho$ such that for all $n\geq 0$ and all $\sigma$ in $2^{<\mathbb{N}}$ one has:
\begin{equation}\label{eqn:boundedtest}
\nu(V_n\cap [\sigma])\leq 2^{-n} \cdot \rho(\sigma)
\end{equation}
An element $\omega$ \emph{passes the test} if $\omega$ is not in $\bigcap_n V_n$.
\end{defn}
\noindent One can show that $\omega$ is in $\mathsf{CR}^{\nu}$ iff it passes all bounded sequential Martin-L\"of $\nu$-tests. The proof in \cite[\S{7.1.4}]{Downey2010aa} generalizes straightforwardly to all computable full support probability measures $\nu$ on Cantor space. However, unlike the rest of the paper, one allows the auxiliary $\rho$ in (\ref{eqn:boundedtest}) to not have full support.

The following is an important property of $\nu\abcon_{bdc}\mu$, and illustrates how this absolute continuity notion interacts with the sequential test notion for computable randomness:
\begin{prop}\label{prop:mixcomp}
If $\nu\abcon_{bdc}\mu$ then $\mathsf{CR}^{\nu}\subseteq \mathsf{CR}^{\mu}$.
\end{prop}
\begin{proof}
Suppose that $\nu\abcon_{bdc}\mu$. Let $f(\omega):=\sup_m \frac{\nu(\omega \upharpoonright m)}{\mu(\omega \upharpoonright m)}$, and note that by considering $m=0$ we have $f\geq 1$ everywhere. Further, note that $f$ is an $L_1(\nu)$ Schnorr test, and hence $\int_{[\sigma]} f \; d\nu$ is computable, uniformly in~$\sigma$. Define computable probability measure $\rho(\sigma)=\nicefrac{\int_{[\sigma]} f \; d\nu}{\int f \; d\nu}$, which has full support since $f\geq 1$ everywhere. Define c.e. open $V_n=\{\omega: f(\omega)>2^n\}$. Then $\nu(V_n\cap [\sigma])\leq 2^{-n} \int_{[\sigma]} f \; d\nu = 2^{-n}\cdot \rho(\sigma)\cdot \mathbb{E}_{\nu}[f]$ for all $\sigma$. Then some computable subsequence of the $V_n$ is a bounded sequential Martin-L\"of test, and hence all $\omega$ in $\mathsf{CR}^{\nu}$ satisfy $f(\omega)\leq 2^n$ for some $n\geq 0$. Hence by Proposition~\ref{crabcon} we have that $\mathsf{CR}^{\nu}\subseteq \mathsf{CR}^{\mu}$.
\end{proof}

Finally, we can prove the following result, mentioned in the introduction \S\ref{sec:intro}.
\thmcrzero*
\begin{proof}
Suppose that $\omega$ in both $\mathsf{Mild}^{\nu}$ and $\mathsf{CR}^{\nu}$. Suppose that $\mu$ is such that $\nu\abcon_{bdc} \mu$. Let $M(\sigma)=\frac{\mu(\sigma)}{\nu(\sigma)}$ be the non-negative computable dyadic $\nu$-martingale. Since $\omega$ in $\mathsf{CR}^{\nu}$, we have that $\lim_n M(\omega\upharpoonright n)$ exists. By $\nu\abcon_{bdc} \mu$ and Proposition~\ref{prop:mixcomp}, we have $\omega$ in $\mathsf{CR}^{\mu}$. Then by Proposition~\ref{crabcon} we have that $\lim_n M(\omega\upharpoonright n)$ is non-zero. Then $\lim_n \frac{M(\omega\upharpoonright (n+1))}{M(\omega\upharpoonright n)}=1$.
The conclusion follows from  Proposition~\ref{prop:equivalentliminf}.
\end{proof}

The next result fills in one last component of the Diagram~\ref{eqn:abcondiagram} in \S\ref{sec:intro}:
\begin{prop}\label{prop:bndimplieskl}
$\nu\abcon_{bd}\mu$ implies $\nu\abcon_{kl}\mu$.
\end{prop}
\begin{proof}
Suppose that $\nu\abcon_{bd}\mu$. Then $\mathbb{E}_{\nu}[\sup_n \frac{\nu(\cdot \upharpoonright n)}{\mu(\cdot \upharpoonright n)}]$ is finite. We must show that $\sup_n\mathbb{E}_{\nu}[\ln \frac{\nu(\omega\upharpoonright n)}{\mu(\omega\upharpoonright n)}]$ is finite (where recall from Example~\ref{ex:oursubmartingale} that this expectation is always non-negative). One has $\ln \frac{\nu(\omega\upharpoonright n)}{\mu(\omega\upharpoonright n)} \leq \frac{\nu(\omega\upharpoonright n)}{\mu(\omega\upharpoonright n)}-1$ and hence similarly with their expectations we have $\mathbb{E}_{\nu}[\ln \frac{\nu(\omega\upharpoonright n)}{\mu(\omega\upharpoonright n)}] \leq \mathbb{E}_{\nu}[\frac{\nu(\omega\upharpoonright n)}{\mu(\omega\upharpoonright n)}]-1$ and $\sup_n \mathbb{E}_{\nu}[\ln \frac{\nu(\omega\upharpoonright n)}{\mu(\omega\upharpoonright n)}]\leq \big( \sup_n\mathbb{E}_{\nu}[\frac{\nu(\omega\upharpoonright n)}{\mu(\omega\upharpoonright n)}]\big)-1 \leq \big( \mathbb{E}_{\nu}[\sup_n \frac{\nu(\omega\upharpoonright n)}{\mu(\omega\upharpoonright n)}]\big)-1<\infty$.
\end{proof}

We close by noting something mentioned immediately after the Diagram~\ref{eqn:abcondiagram} in \S\ref{sec:intro}, namely that the derivative $\frac{d\nu}{d\mu}$ being a computable point of the computable Polish space $L_2(\mu)$ (cf. \cite[\S{2.3}]{Huttegger2024-zb}) is sufficient for $\nu\abcon_{bdc}\mu$ and $\nu\abcon_{comp}\mu$. We do not know whether it is sufficient for  $\nu\abcon_{klc}\mu$.
\begin{prop}\label{prop:iamgood}
Suppose that  $\frac{d\nu}{d\mu}$ is $L_2(\mu)$-computable. Then both $\nu\abcon_{bdc}\mu$ and $\nu\abcon_{comp}\mu$. 
\end{prop}
\begin{proof}
Suppose that $\frac{d\nu}{d\mu}$ is $L_2(\mu)$-computable. 
 
First we show $\nu\abcon_{bdc}\mu$. Recall that $\mathbb{E}_{\mu}[\frac{d\nu}{d\mu}\mid\mathscr{F}_n](\omega)=\frac{\nu(\omega\upharpoonright n)}{\mu(\omega\upharpoonright n)}$ for all $\omega$ in $2^{\mathbb{N}}$. Since the maximal function is computable continuous from $L_p(\mu)$ to $L_p(\mu)$ when $p>1$ is computable, the function $f(\omega):=\sup_m \frac{\nu(\omega \upharpoonright m)}{\mu(\omega \upharpoonright m)}$ is $L_2(\mu)$ computable \cite[Proposition 5.2]{Huttegger2024-zb}. Further, one has that $f$ is $L_1(\nu)$-computable. This last point is because the inclusion map is computable continuous from $L_2(\mu)$ to $L_1(\nu)$ (note the change from $\mu$ to~$\nu$). For, the inclusion map has a computable uniform modulus of continuity since it is $\alpha$-Lipschitz, where $\alpha:=\|\frac{d\nu}{d\mu}\|_{L_2(\mu)}$, which is computable by hypothesis of $\frac{d\nu}{d\mu}$ being $L_2(\mu)$-computable. In particular, H\"older implies that it is $\alpha$-Lipschitz: $\|g\|_{L_1(\nu)} = \|g\cdot \frac{d\nu}{d\mu} \|_{L_1(\mu)}\leq \alpha\cdot \|g\|_{L_2(\mu)}$, for any $g$ in $L_2(\mu)$. Moreover, the countable dense set of $L_2(\mu)$ is $L_1(\nu)$-computable. To see this, it suffices in turn to show that if $U$ is c.e. open and $\mu(U)$ is computable, then $\nu(U)=\int_U \frac{d\nu}{d\mu} \; d\mu$ is computable. And this then follows by noting that the map $h\mapsto I_U\cdot h$ is computable continuous from $L_1(\mu)$ to itself: it is 1-Lipschitz and sends the countable dense set to a uniformly computable sequence of points in $L_1(\mu)$.

Second we show $\nu\abcon_{comp}\mu$. Since  $\frac{d\nu}{d\mu}$ is $L_2(\mu)$-computable, it is  also $L_1(\mu)$-computable.\footnote{In general, if $p\leq q$ are computable, then any computable point of $L_q(\mu)$ is also a computable point of $L_p(\mu)$. Cf. \cite[\S{A3}]{Rute2012aa}, \cite[50]{FreHanNiesSte2014}, \cite[\S{2.3}]{Huttegger2024-zb}.} Choose a computable sequence $f_n\geq 0$ from the countable dense set of $L_1(\mu)$ such that $f_n\rightarrow \frac{d\nu}{d\mu}$ fast. Without loss of generality $f_n=\sum_{i=1}^{k_n} q_{n,i}\cdot I_{[\sigma_{n,i}]}$, and so we define $q_n = \sup_{1\leq i\leq k_n} \left|q_{n,i}\right|$. For each rational $\epsilon>0$, compute $n\geq 0$ such that $2^{-n}<\frac{\epsilon}{2}$ and define $m(\epsilon)=\frac{1}{2}\cdot \frac{\epsilon}{q_n}$ (or just $\frac{1}{2}\cdot \epsilon$ if $q_n=0$). If $\mu(A)<m(\epsilon)$, then one has that $\nu(A)=\int_A \frac{d\nu}{d\mu} \; d\mu \leq 2^{-n} + \int_A f_n \; d\mu <\frac{\epsilon}{2} + \mu(A)\cdot q_n<\epsilon$.
\end{proof}

\section{Medium horizon and proof of Theorem~\ref{thm:main:mlr:weakmiddle}}\label{sec:middleground}

So far, we have focused on weak merging where we consider the one-step merging horizon $\mathscr{F}_{n+1}$ (cf. Definition~\ref{defn:weakvsstrong}). But it is also natural to consider the two-step merging horizon $\mathscr{F}_{n+2}$ and more generally the $\ell$-step merging horizon $\mathscr{F}_{n+\ell}$ for $\ell>1$. In this section we prove Theorem~\ref{thm:main:mlr:weakmiddle}, stated in \S\ref{sec:intro}, of these merging horizons. Indeed we prove a generalization of it, which isolates relevant aspects of the function $n\mapsto n+\ell$.

We begin with some notation. Any injective function $g:\mathbb{N}\rightarrow \mathbb{N}$ induces an equivalence relation $\sim_g$ on $\mathbb{N}$ by $n\sim_g m$ iff there is $k\geq 0$ such that $g^{(k)}(n)=m$ or $g^{(k)}(m)=n$, where $g^{(k)}$ denotes the $k$-fold iterate of $g$, and where $g^{(0)}$ is the identity function. We write the equivalence class of $n$ as $[n]_g$, or just as $[n]$ when $g$ is clear from context.

This notation in place, we define:
\begin{defn}
A merging horizon $\mathscr{G}_n$ is said to be \emph{augmented} if there is a computable function $g:\mathbb{N}\rightarrow \mathbb{N}$ such that $n<g(n)<g(n+1)$ for all $n\geq 0$ and such that $\mathscr{G}_n=\mathscr{F}_{g(n)}$.

It is said to be \emph{finitely augmented} if the function $g$ is such that $\mathbb{N}$ has only finitely many equivalence classes under $\sim_g$.
\end{defn}

The paradigmatic example is the following, where of course the case of e.g. $\ell=2$ just is the case of the merging horizon $\mathscr{F}_{n+2}$.
\begin{ex}\label{ex:gellmap}
Let $\ell\geq 1$ be fixed, and let $g_{\ell}(n)=n+\ell$. Then  $\mathbb{N}$ has exactly $\ell$ equivalence classes under $\sim_{g_{\ell}}$, namely $\{ \{ n+k\cdot \ell: k\geq 0\}: n<\ell\}$.
\end{ex}
\noindent Of course, these equivalence classes are similar to the elements of the group $\nicefrac{\mathbb{Z}}{\ell \mathbb{Z}}$, but as far as we are aware that is just a coincidence.

Suppose $\mathscr{G}_n$ is augmented, and its set of equivalence classes is $\{ [n_i] :i\in I\}$, where the displayed $n_i$ is the least element of its equivalence class $[n_i]$. Since $g(n)>n$ for all $n\geq 0$, each equivalence class is infinite, and $[n_i]$ can in turn be enumerated as $n_i <g(n_i)<g^{(2)}(n_i)<g^{(3)}(n_i)<\cdots$, which generates a filtration $\mathscr{F}_{n_i}\subseteq \mathscr{F}_{g(n_i)}\subseteq \mathscr{F}_{g^{(2)}(n_i)}\subseteq \mathscr{F}_{g^{(3)}(n_i)}\subseteq \cdots$. This is different than the filtration $\mathscr{G}_n:=\mathscr{F}_{g(n)}$, since the latter does not include the least element $\mathscr{F}_{n_i}$, and since the equivalence class $[n_i]$ may not be all of $\mathbb{N}$. Note that distinct equivalence classes $[n_i], [n_j]$ will generate disjoint filtrations. This is easiest to see by way of example:

\begin{ex}\label{ex:gellmap2}
Suppose that $\ell=2$, and consider $g_{\ell}$ as in Example~\ref{ex:gellmap}. Then the two equivalence classes are just the evens and the odds, generating two filtrations:
\begin{align}
& \mathscr{F}_0 \subseteq \mathscr{F}_2 \subseteq \mathscr{F}_4 \subseteq \cdots \notag \\
& \mathscr{F}_1 \subseteq \mathscr{F}_3 \subseteq \mathscr{F}_5 \subseteq \cdots \label{ex:gellmap2:subseq}
\end{align}
Note by contrast that the merging horizon $\mathscr{G}_0\subseteq \mathscr{G}_1\subseteq \mathscr{G}_2\subseteq\cdots$ associated to $g_{\ell}$ is $\mathscr{F}_2 \subseteq \mathscr{F}_3 \subseteq \mathscr{F}_4 \subseteq\cdots$, which does not include $\mathscr{F}_0, \mathscr{F}_1$.
\end{ex}

In what follows, we generalize the work of \S\ref{sec:proof:main} to the finitely augmented setting. Basically this comes down to partitioning as in (\ref{ex:gellmap2:subseq}) and performing the analysis of \S\ref{sec:proof:main} on the individual parts. Since the proofs are largely the same as the corresponding parts of \S\ref{sec:proof:main}, we omit many of the details.

First, in parallel to Proposition~\ref{prop:derivativecalc}, we have:
\begin{prop}
For $n\geq 0$ and $m>n$ one has:
\begin{equation}
\frac{d(\nu(\cdot \mid \mathscr{F}_n)(\omega) \upharpoonright \mathscr{F}_{m})}{d(\mu(\cdot \mid \mathscr{F}_n)(\omega) \upharpoonright \mathscr{F}_{m})}(\omega^{\prime}) = \sum_{\left|\sigma\right|=m\mbox{-}n} \frac{\nu((\omega\upharpoonright n)\,\sigma\mid \omega\upharpoonright n\,)}{\mu((\omega\upharpoonright n)\,\sigma\mid \omega\upharpoonright n\,)} \cdot X_n^{\sigma}(\omega^{\prime}) \label{eqn:thedervweak2}
\end{equation}
where
\begin{equation}
X_n^{\sigma}(\omega^{\prime})=
\begin{cases}
1      & \text{if $\forall \; i<m-n \;\; \omega^{\prime}(n+i)=\sigma(i)$}, \\
0      & \text{otherwise}.
\end{cases}
\end{equation}
\end{prop}
\begin{proofdetail}
\begin{proof}
To show~(\ref{eqn:thedervweak2}), we show that the expression on the right-hand side of it satisfies the defining property of the Radon-Nikodym derivative on the left-hand side of it. As an initial step, consider an arbitrary Borel event $A$, and consider the following, wherein $\frac{d(\nu(\cdot \mid \mathscr{F}_n)(\omega) \upharpoonright \mathscr{F}_{m})}{d(\mu(\cdot \mid \mathscr{F}_n)(\omega) \upharpoonright \mathscr{F}_{m})}(\omega^{\prime})$ means the formula on the right-hand side of (\ref{eqn:thedervweak2}):
\begin{align}
& \int_A \frac{d(\nu(\cdot \mid \mathscr{F}_n)(\omega) \upharpoonright \mathscr{F}_{m})}{d(\mu(\cdot \mid \mathscr{F}_n)(\omega) \upharpoonright \mathscr{F}_{m})}(\omega^{\prime}) \;d\mu(\cdot \mid \omega\upharpoonright n\,)(\omega^{\prime}) \notag \\
= & \frac{1}{\mu(\omega\upharpoonright n\,)} \int_{A\cap [\omega\upharpoonright n\,]} \frac{d(\nu(\cdot \mid \mathscr{F}_n)(\omega) \upharpoonright \mathscr{F}_{m})}{d(\mu(\cdot \mid \mathscr{F}_n)(\omega) \upharpoonright \mathscr{F}_{m})}(\omega^{\prime})\; d\mu(\omega^{\prime}) \notag \\
= & \sum_{\left|\sigma\right|=m\mbox{-}n} \frac{1}{\mu(\omega\upharpoonright n\,)} \int_{A\cap [\omega\upharpoonright n]} \frac{\nu((\omega\upharpoonright n)\,\sigma\mid \omega\upharpoonright n\,)}{\mu((\omega\upharpoonright n)\,\sigma\mid \omega\upharpoonright n\,)} \cdot X_n^{\sigma}(\omega^{\prime})  \; d\mu(\omega^{\prime}) \notag \\
=&  \sum_{\left|\sigma\right|=m\mbox{-}n}  \frac{1}{\mu(\omega\upharpoonright n\,)}  \frac{\nu((\omega\upharpoonright n)\,\sigma\mid \omega\upharpoonright n\,)}{\mu((\omega\upharpoonright n)\,\sigma\mid \omega\upharpoonright n\,)} \mu(A\cap [(\omega\upharpoonright n)\,\sigma]) \notag \\ =& \sum_{\left|\sigma\right|=m\mbox{-}n} \nu((\omega\upharpoonright n)\,\sigma\mid \omega\upharpoonright n\,) \mu(A\mid (\omega\upharpoonright n)\,\sigma) \label{eqn:derivative2}
\end{align}
If $A$ is in $\mathscr{F}_m$, then we have $A=\bigsqcup_{i=1}^k [\tau_i\,\sigma_i]$ where $\left|\tau_i\right|=n$ and $\left|\sigma_i\right|=m\mbox{-}n$. Then one has 
\begin{align*}
& \sum_{\left|\sigma\right|=m\mbox{-}n} \nu((\omega\upharpoonright n)\,\sigma\mid \omega\upharpoonright n\,) \mu(A\mid (\omega\upharpoonright n)\,\sigma) \\
= & \sum_{\left|\sigma\right|=m\mbox{-}n} \nu((\omega\upharpoonright n)\,\sigma\mid \omega\upharpoonright n\,) \sum_{i=1}^k \mu( \tau_i\,\sigma_i\mid (\omega\upharpoonright n)\,\sigma) \\
= & \sum_{i=1}^k \sum_{\left|\sigma\right|=m\mbox{-}n} \nu((\omega\upharpoonright n)\, \sigma \mid \omega\upharpoonright n) \mu(\tau_i\,\sigma_i \mid (\omega\upharpoonright n)\, \sigma) \\
= & \sum_{\tau_i=\omega\upharpoonright n, \sigma_i=\sigma} \nu(\tau_i\, \sigma_i\mid \omega\upharpoonright n) = \nu(A\mid \omega\upharpoonright n)
\end{align*}
which is what we wanted to show.
\end{proof}
\end{proofdetail}

Second, corresponding to Proposition~\ref{prop:calculatekullback}, we have:
\begin{prop}\label{prop:calculatekullback2}
Suppose that  $\mathscr{G}_n$ is augmented with $g$. Then one has
\begin{equation}
 D_{\mathscr{G}_n}(\nu \mid \mu)(\omega)= \sum_{\left|\sigma\right|=g(n)-n} \bigg(\ln \frac{\nu((\omega\upharpoonright n)\, \sigma \mid \omega\upharpoonright n\,)}{\mu((\omega\upharpoonright n)\, \sigma \mid \omega\upharpoonright n\,)} \bigg) \cdot \nu((\omega\upharpoonright n)\, \sigma \mid \omega\upharpoonright n\,)\label{eqn:dn1v2}
\end{equation}
\end{prop}
\begin{proofdetail}
\begin{proof}
Fix $n\geq 0$ and let $m=g(n)$, so that $m>n$. Then one has
\begin{align*}
  D_{\mathscr{G}_n}(\nu \mid \mu)(\omega)  = &  D(\nu(\cdot \mid \mathscr{F}_n)(\omega) \upharpoonright \mathscr{F}_m \mid  \mu(\cdot \mid \mathscr{F}_n)(\omega)\upharpoonright \mathscr{F}_m)\\
= & \int \ln \frac{d(\nu(\cdot \mid \mathscr{F}_n)(\omega) \upharpoonright \mathscr{F}_m)}{d(\mu(\cdot \mid \mathscr{F}_n)(\omega) \upharpoonright \mathscr{F}_m)}(\omega^{\prime}) \; d\nu(\cdot \mid \omega\upharpoonright n\,)(\omega^{\prime}) \\
= & \frac{1}{\nu(\omega\upharpoonright n)} \int_{[\omega\upharpoonright n]} \ln \frac{d(\nu(\cdot \mid \mathscr{F}_n)(\omega) \upharpoonright \mathscr{F}_m)}{d(\mu(\cdot \mid \mathscr{F}_n)(\omega) \upharpoonright \mathscr{F}_m)}(\omega^{\prime}) \; \; d\nu(\omega^{\prime}) \\
= & \sum_{\left|\sigma\right|=m\mbox{-}n} \frac{1}{\nu(\omega\upharpoonright n)} \int_{[(\omega\upharpoonright n)\,\sigma]} \ln  \frac{d(\nu(\cdot \mid \mathscr{F}_n)(\omega) \upharpoonright \mathscr{F}_m)}{d(\mu(\cdot \mid \mathscr{F}_n)(\omega) \upharpoonright \mathscr{F}_m)}(\omega^{\prime})\; d\nu(\omega^{\prime}) \\
= & \sum_{\left|\sigma\right|=m\mbox{-}n} \frac{1}{\nu(\omega\upharpoonright n)} \int_{[(\omega\upharpoonright n) \,\sigma]}  \ln  \frac{\nu((\omega\upharpoonright n)\,\sigma\mid \omega\upharpoonright n\,)}{\mu((\omega\upharpoonright n)\,\sigma\mid \omega\upharpoonright n\,)}  \; \; d\nu(\omega^{\prime}) \\
= & \sum_{\left|\sigma\right|=m\mbox{-}n} \bigg(\ln \frac{\nu((\omega\upharpoonright n)\, \sigma \mid \omega\upharpoonright n\,)}{\mu((\omega\upharpoonright n)\, \sigma \mid \omega\upharpoonright n\,)} \bigg) \cdot \nu((\omega\upharpoonright n)\, \sigma \mid \omega\upharpoonright n\,)
\end{align*}
In this, it is only in the penultimate step that we appeal to (\ref{eqn:thedervweak2}), where the expression in (\ref{eqn:thedervweak2}) simplifies due to our being in the the clopen $[(\omega\upharpoonright n)\, \sigma]$. 
\end{proof}
\end{proofdetail}

The correspondence between classical martingales and dyadic martingales was discussed in Definition~\ref{defn:dyadicmartingale}. We need to slightly update this in order to handle such filtrations as those in (\ref{ex:gellmap2:subseq}). Hence parallel to Definition~\ref{defn:dyadic} we define:
\begin{defn}\label{defn:dyadicv2}
Suppose that $\Delta$ is a computable set of natural numbers enumerated in increasing order as $\delta_0<\delta_1<\cdots$.  

Define $2^{<\Delta} = \{ \sigma \in 2^{<\mathbb{N}}: \exists \; k\geq 0 \; \left|\sigma\right|=\delta_k\}$.

A real-valued function is called \emph{$\Delta$-ary} if it has domain $2^{<\Delta}$.  

A $\Delta$-ary function $F:2^{<\Delta}\rightarrow \mathbb{R}$ is \emph{computable} if $F(\sigma)$ is a computable real number, uniformly in $\sigma$ from $2^{<\Delta}$.

If $F:2^{<\Delta}\rightarrow \mathbb{R}$ is a $\Delta$-ary function, then we define the sequence of functions $F_k:2^{\mathbb{N}}\rightarrow \mathbb{R}$ by $F_k(\omega)=F(\omega\upharpoonright \delta_k)$.
\end{defn}

In this last part of the definition, note that $F_k(\omega)$ is defined in terms of the first $\delta_k$ bits of $\omega$, rather than the first $k$ bits of $\omega$.

Parallel to Definition~\ref{defn:dyadicmartingale}, we define: 
\begin{defn}
Suppose that $\Delta$ is a computable set of natural numbers enumerated in increasing order as $\delta_0<\delta_1<\cdots$.  
Then a \emph{$\Delta$-ary $\nu$-martingale} is a $\Delta$-ary function $M^{\Delta}:2^{<\Delta}\rightarrow \mathbb{R}$ satisfying the following for all $k\geq 0$ and all $\tau$ of length~$\delta_k$:
\begin{equation}\label{eqn:martingaleconditionv2}
M^{\Delta}(\tau) = \sum_{\left|\sigma\right|=\delta_{k+1}-\delta_k} M^{\Delta}(\tau\,\sigma)\nu(\tau\,\sigma\mid \tau)
\end{equation}
\noindent Further, {\it $\Delta$-ary $\nu$-submartingales} are defined the same, but with equality replaced by $\leq$; and {\it $\Delta$-ary $\nu$-supermartingales} are defined just the same but with the equality replaced by $\geq$.
\end{defn}

Corresponding to $\Delta$, consider the increasing filtration $\mathscr{D}_k=\mathscr{F}_{\delta_k}$. Using the preferred version of the conditional expectation from (\ref{eqn:preferredconditional}), one can check that for $M:2^{<\Delta}\rightarrow \mathbb{R}$, one has that (\ref{eqn:martingaleconditionv2}) is equivalent to the classical $\nu$-martingale condition $M_k^{\Delta} = \mathbb{E}_{\nu}[M_{k+1}^{\Delta}\mid \mathscr{D}_k]$. Likewise the  $\Delta$-ary $\nu$-submartingale condition can be rephrased as $M_k^{\Delta}\leq \mathbb{E}_{\nu}[M_{k+1}^{\Delta}\mid \mathscr{D}_k]$ and the $\Delta$-ary $\nu$-supermartingale condition can be rephrased as $M_k^{\Delta}\geq \mathbb{E}_{\nu}[M_{k+1}^{\Delta}\mid \mathscr{D}_k]$. To see these equivalences it suffices to note that for $\omega$ in $[\tau]$, where $\tau$ is of length $\delta_k$, we have:
\begin{align*}
& \mathbb{E}_{\nu}[M_{k+1}^{\Delta}\mid \mathscr{D}_k](\omega) = \mathbb{E}_{\nu}[M_{k+1}^{\Delta}\mid \mathscr{F}_{\delta_k}](\tau^{\frown}\overline{0}) =\frac{1}{\nu(\tau)} \int_{[\tau]} M_{k+1}^{\Delta}(\omega^{\prime}) \; d\nu(\omega^{\prime}) \\
& \sum_{\left|\sigma\right|=\delta_{k+1}-\delta_k} \frac{1}{\nu(\tau)} \int_{[\tau\,\sigma]} M^{\Delta}(\tau\, \sigma) \; d\nu(\omega^{\prime}) =  \sum_{\left|\sigma\right|=\delta_{k+1}-\delta_k} M^{\Delta}(\tau\,\sigma)\nu(\tau\,\sigma\mid \tau)
\end{align*}

As with Proposition~\ref{prop:jensensubmartingal}, the main example of $\Delta$-ary $\nu$-submartingales comes from convex functions applied to $\Delta$-ary $\nu$-martingales.

In parallel to the Effective Doob Decomposition Theorem~\ref{thm:doobdecomp}, we have:
\begin{thm}\label{thm:doobdecomp2} (Effective $\Delta$-ary Doob Decomposition).

Suppose that $\Delta$ is a computable set of natural numbers enumerated in increasing order as $\delta_0<\delta_1<\cdots$.  

Suppose that $L^{\Delta}$ is a computable $\Delta$-ary $\nu$-submartingale. Then there is a unique pair of computable $\Delta$-ary $\nu$-martingale $N^{\Delta}$ and an increasing non-negative computable $\Delta$-ary function $A^{\Delta}$ with the following properties, for all $k\geq 0$ and $\tau$ of length $\delta_k$ and $\sigma$ of length $\delta_{k+1}-\delta_k$:
\begin{enumerate}[leftmargin=*]
    \item \label{eqn:apartofdoob:02} $L^{\Delta}(\tau)=N^{\Delta}(\tau)+A^{\Delta}(\tau)$
    \item \label{eqn:apartofdoob:002} $N^{\Delta}(\tau)=L^{\Delta}(\tau)$ if $k=0$
    \item \label{eqn:apartofdoob2} $A^{\Delta}(\tau\sigma)-A^{\Delta}(\tau) = \big(\sum_{\left|\rho\right|=\delta_{k+1}-\delta_k} L^{\Delta}(\tau \rho) \nu(\tau \rho \mid \tau) \big) - L^{\Delta}(\tau)$
\end{enumerate}
\end{thm}
In parallel to Theorem~\ref{thm:doobdecomp}, we refer to $A^{\Delta}$ the \emph{$\Delta$-ary predictable process}, since (\ref{eqn:apartofdoob2}) shows that $A^{\Delta}(\tau\sigma)$ depends on $\tau$ and not~$\sigma$. Note that the identity in (\ref{eqn:apartofdoob2}) can also be expressed in classical terms as follows, where $\omega$ ranges over $2^{\mathbb{N}}$ and $k\geq 0$:
\begin{equation}\label{eqn:classicalapartofdoob2}
A_{k+1}^{\Delta}(\omega)-A_k^{\Delta}(\omega) = \mathbb{E}_{\nu}[L_{k+1}^{\Delta}\mid \mathscr{D}_k](\omega) - L_k^{\Delta}(\omega)
\end{equation}

\begin{proofdetail}
\begin{proof}
Uniqueness of $A^{\Delta}$ follows by an induction on $k$ in the length $\delta_k$ of $\tau$, with the base case being taken care of by (\ref{eqn:apartofdoob:0})-(\ref{eqn:apartofdoob:00}) and the induction step being taken of by  (\ref{eqn:apartofdoob}); and then uniqueness of $N^{\Delta}$ follows from (\ref{eqn:apartofdoob:0}). For existence, we first define an auxiliary computable $\Delta$-ary function $H^{\Delta}$. Define $H^{\Delta}(\tau) = L(\tau)$ if $\tau$ is of length $\delta_0$. If $\tau$ is of length $\delta_k$ and $\sigma$ is of length $\delta_{k+1}-\delta_k$ then define: 
\begin{equation*}
H^{\Delta}(\tau\,\sigma) = L^{\Delta}(\tau\,\sigma)-
\big(\sum_{\left|\rho\right|=\delta_{k+1}-\delta_k} L^{\Delta}(\tau\,\rho) \nu(\tau\,\rho \mid \tau) \big)
\end{equation*}
Since $\sigma$ does not appear after the minus sign in $H^{\Delta}(\tau\,\sigma)$, we have the following:
\begin{equation}\label{eqn:hproperty2}
\sum_{\left|\rho\right|=\delta_{k+1}-\delta_k} H^{\Delta}(\tau \rho) \nu(\tau \rho\mid \tau)=0
\end{equation}

Then we define computable $\Delta$-ary functions $N^{\Delta},A^{\Delta}$ as follows, where $\tau$ is of length~$\delta_k$:
\begin{equation*}
N^{\Delta}(\tau)=\sum_{j\leq k} H^{\Delta}(\tau\upharpoonright \delta_j), \hspace{10mm} A^{\Delta}(\tau)=L^{\Delta}(\tau)-N^{\Delta}(\tau)
\end{equation*}

Then, by (\ref{eqn:hproperty2}), $N^{\Delta}$ is a $\Delta$-ary $\nu$-martingale:
\begin{equation*}
\sum_{\left|\rho\right|=\delta_{k+1}-\delta_k} N^{\Delta}(\tau\,\rho) \nu(\tau\,\rho\mid \tau) = N^{\Delta}(\tau)+\sum_{\left|\rho\right|=\delta_{k+1}-\delta_k} H^{\Delta}(\tau \rho) \nu(\tau \rho\mid \tau)= N^{\Delta}(\tau)
\end{equation*}

Finally, one has the following, where the inequality at the end comes from $L$ being a $\Delta$-ary $\nu$-submartingale: 
\begin{align*}
&A^{\Delta}(\tau\,\sigma)-A^{\Delta}(\tau) = L^{\Delta}(\tau\,\sigma)-N^{\Delta}(\tau\,\sigma)-\big( L^{\Delta}(\tau)-N^{\Delta}(\tau) \big) \\
& = L^{\Delta}(\tau\,\sigma)-L^{\Delta}(\tau)-\big( N^{\Delta}(\tau\,\sigma)-N^{\Delta}(\tau) \big)  \\
& = L^{\Delta}(\tau\,\sigma)-L^{\Delta}(\tau)-H^{\Delta}(\tau\,\sigma) \\
& = \big(\sum_{\left|\rho\right|=\delta_{k+1}-\delta_k} L^{\Delta}(\tau \rho) \nu(\tau \rho \mid \tau) \big) - L^{\Delta}(\tau) \geq 0 
\end{align*}
\end{proof}
\end{proofdetail}

Finally, we have the following in parallel to Theorem~\ref{prop:klandtheincrements}:
\begin{thm}\label{prop:klandtheincrements2}
Suppose that $\mathscr{G}_n$ is augmented with $g$. 

Suppose that an $\sim_g$-equivalence class $\Delta=[\delta_0]$ is fixed with least member $\delta_0$ and enumerated as $\delta_0<\delta_1<\cdots$.  

Define $\Delta$-ary computable $\nu$-submartingale by $L^{\Delta}(\tau)=-\ln \frac{\mu(\tau)}{\nu(\tau)}$, and let $A^{\Delta}$ be the 
$\Delta$-predictable process from the $\Delta$-effective Doob decomposition of $L^{\Delta}$. Then for all $k\geq 0$ one has
\begin{equation}\label{eqn:klsub2}
 D_{\mathscr{G}_{\delta_k}}(\nu \mid \mu)(\omega)   = A^{\Delta}_{k+1}(\omega)-A^{\Delta}_k(\omega)
\end{equation}
\end{thm}
\begin{proofdetail}
\begin{proof}
Let $m=\delta_k$ and let $\sigma$ be of length $\delta_{k+1}-\delta_k$. Then one has: 
\begin{align*}
& \ln \frac{\nu((\omega\upharpoonright m)\, \sigma \mid \omega\upharpoonright m\,)}{\mu((\omega\upharpoonright m)\, \sigma \mid \omega\upharpoonright m\,)}=-\ln \frac{\mu((\omega\upharpoonright m)\, \sigma \mid \omega\upharpoonright m\,)}{\nu((\omega\upharpoonright m)\, \sigma \mid \omega\upharpoonright m\,)} \\
& =-\bigg(\ln \mu((\omega\upharpoonright m)\, \sigma \mid \omega\upharpoonright m\,) - \ln \nu((\omega\upharpoonright m)\, \sigma \mid \omega\upharpoonright m\,)\bigg) \\
&= -\bigg(\ln \mu((\omega\upharpoonright m)\, \sigma)-\ln \mu((\omega\upharpoonright m)\,)-\big(\ln \nu((\omega\upharpoonright m)\, \sigma)-\ln \nu(\omega\upharpoonright m\,)\big)\bigg) \\
&= -\ln \mu((\omega\upharpoonright m)\, \sigma)+\ln \mu(\omega\upharpoonright m\,)+\ln \nu((\omega\upharpoonright m)\, \sigma)-\ln \nu(\omega\upharpoonright m\,) \\
&= -\big(\ln \mu((\omega\upharpoonright m)\, \sigma)-\ln \nu((\omega\upharpoonright m)\, \sigma)\big)+\big(\ln \mu((\omega\upharpoonright m)\,)-\ln \nu(\omega\upharpoonright m\,)\big) \\
& = -\ln M^{\Delta}((\omega\upharpoonright m)\,\sigma) + \ln M^{\Delta}(\omega\upharpoonright m\,) = L^{\Delta}((\omega\upharpoonright m)\,\sigma)-L^{\Delta}(\omega\upharpoonright m\,)
\end{align*}
where $M^{\Delta}(\tau)=\frac{\mu(\tau)}{\nu(\tau)}$ is the $\Delta$-ary $\nu$-martingale.

Then we get the following by setting $n:=\delta_k$ on the left-hand side of (\ref{eqn:dn1v2}), and by abbreviating $m=\delta_k$ on the right-hand side:
\begin{align*}
 & D_{\mathscr{G}_{\delta_k}}(\nu \mid \mu)(\omega)= \sum_{\left|\sigma\right|=\delta_{k+1}-\delta_k} \bigg(\ln \frac{\nu((\omega\upharpoonright m)\, \sigma \mid \omega\upharpoonright m\,)}{\mu((\omega\upharpoonright m)\, \sigma \mid \omega\upharpoonright m\,)} \bigg) \cdot \nu((\omega\upharpoonright m)\, \sigma \mid \omega\upharpoonright m\,) \\
=& \sum_{\left|\sigma\right|=\delta_{k+1}-\delta_k}  \big( L^{\Delta}((\omega\upharpoonright m)\,\sigma)-L^{\Delta}(\omega\upharpoonright m\,) \big) \cdot \nu((\omega\upharpoonright m)\,\sigma \mid \omega\upharpoonright m\,) \\
= &  \big(\sum_{\left|\sigma\right|=\delta_{k+1}-\delta_k} L^{\Delta}((\omega\upharpoonright m)\, \sigma) \nu((\omega\upharpoonright m)\,\sigma \mid \omega\upharpoonright m) \big) - L^{\Delta}(\omega\upharpoonright m)
\end{align*}
From this last line and Theorem~\ref{thm:doobdecomp2}(\ref{eqn:apartofdoob2}) one gets the desired conclusion (\ref{eqn:klsub2}).
\end{proof}
\end{proofdetail}

\begin{thm}\label{thm:thmmainmiddle}
Suppose that the merging horizon $\mathscr{G}_n$ is finitely augmented. Then $\mathsf{MLR}^{\nu}= \mathsf{MR}^{\nu}_1(\abcon_{kl}, \mathscr{G}_n, D)$ and $\mathsf{SR}^{\nu}\supseteq \mathsf{MR}^{\nu}_1(\abcon_{klc}, \mathscr{G}_n, D)$.
\end{thm}
\begin{proof}
Let $g$ be the computable function with which $\mathscr{G}_n$ is finitely augmented. Let $\{\Delta_i: i<\ell\}$ be its finitely many equivalence classes. For each $i<\ell$ we enumerate $\Delta_i$ in increasing order as $\delta_{i,0}<\delta_{i,1}<\cdots$.

For the backward direction: for all $n\geq 0$, since $g(n)>n$, one has the containment $\mathscr{F}_{n+1}\subseteq \mathscr{F}_{g(n)}=\mathscr{G}_n$. Then by antimonotonicity and Theorem~\ref{thm:main:mlr:weak}, we have $ \mathsf{MR}^{\nu}_1(\abcon_{kl}, \mathscr{G}_n, D)\subseteq \mathsf{MR}^{\nu}_1(\abcon_{kl}, \mathscr{F}_{n+1}, D) =\mathsf{MLR}^{\nu}$. Likewise, we have the containment $ \mathsf{MR}^{\nu}_1(\abcon_{klc}, \mathscr{G}_n, D)\subseteq \mathsf{MR}^{\nu}_1(\abcon_{klc}, \mathscr{F}_{n+1}, D) =\mathsf{SR}^{\nu}$.

It remains to prove the forward direction. Suppose that $\omega$ is in $\mathsf{MLR}^{\nu}$. Suppose that $\nu\abcon_{kl} \mu$. Let $L(\sigma)=-\ln \frac{\mu(\sigma)}{\nu(\sigma)}$ be the dyadic $\nu$-submartingale. Then $\sup_n \mathbb{E}_{\nu} L_n<\infty$. For each $i<\ell$, let $L^{\Delta_i}(\sigma)=-\ln \frac{\mu(\sigma)}{\nu(\sigma)}$ be the $\Delta_i$-ary $\nu$-submartingale. For each $i<\ell$, let $L^{\Delta_i}=N^{\Delta_i}+A^{\Delta_i}$ be the effective $\Delta_i$-ary Doob decomposition. We have $\mathbb{E}_{\nu} A_k^{\Delta_i}=\mathbb{E}_{\nu} L_k^{\Delta_i} -\mathbb{E}_{\nu} N^{\Delta_i}_k = \mathbb{E}_{\nu} L_k^{\Delta_i}- \mathbb{E}_{\nu} N_0^{\Delta_i}$. Then:
\begin{equation}\label{eqn:lastone}
\sup_k \mathbb{E}_{\nu} A_k^{\Delta_i}= \big(\sup_k \mathbb{E}_{\nu} L_k^{\Delta_i}\big)- \mathbb{E}_{\nu} N_0^{\Delta_i} \leq  \big(\sup_k \mathbb{E}_{\nu} L_k\big)- \mathbb{E}_{\nu} N_0^{\Delta_i}<\infty
\end{equation}
Let $f_i(\omega)=\lim_k A_k^{\Delta_i}(\omega) = \sup_k A^{\Delta_i}_k(\omega)$. By the Monotone Convergence Theorem, we have $\mathbb{E}_{\nu} f_i =\lim_k \mathbb{E}_{\nu} A_k^{\Delta_i}=\sup_k \mathbb{E}_{\nu} A_k^{\Delta_i}<\infty$. Then $f_i$ is a Martin-L\"of $L_1(\nu)$ test. Since $\omega$ is in $\mathsf{MLR}^{\nu}$, we have $f_i(\omega)<\infty$. Further, since $A_k^{\Delta_i}$ is non-negative and increasing and $A_0^{\Delta_i}(\omega)=0$, by telescoping we have $f_i(\omega)=\sum_k (A_{k+1}^{\Delta_i}(\omega)-A_k^{\Delta_i}(\omega))$. Then by Theorem~\ref{prop:klandtheincrements2} we have
\begin{align}
& \sum_n D_{\mathscr{G}_n}(\nu \mid \mu)(\omega) = \sum_{i<\ell} \sum_k D_{\mathscr{G}_{\delta_{i,k}}}(\nu\mid \mu)(\omega) \notag \\
& =  \sum_{i<\ell} \sum_k \big(A^{\Delta_i}_{k+1}(\omega)-A^{\Delta_i}_k(\omega)\big) = \sum_{i<\ell} f_i(\omega) <\infty \label{eqn:lastlastone}
\end{align}
\end{proof}

Theorem~\ref{thm:main:mlr:weakmiddle} follows from Theorem \ref{thm:thmmainmiddle} by applying  it to Example~\ref{ex:gellmap}. It is not obvious whether the stated inclusion in these theorems involving Schnorr randomness can be improved to an identity. To secure this, one would have to find a way to ensure that the quantity on the very left-hand side of (\ref{eqn:lastone}) is computable in addition to being finite. 

Further, it is not obvious how to improve Theorem~\ref{thm:main:mlr:weakmiddle} to handle all augmented merging horizons, rather than just the finitely augmented ones. The previous proof, particularly the last step in (\ref{eqn:lastlastone}), stops working once the number of equivalence classes becomes infinite.

\bibliographystyle{amsalpha}
\bibliography{01-HWZ-project-02-weak-merging}

\end{document}